\documentclass{mcom-l}
\usepackage{amssymb}
\usepackage{graphicx}
\usepackage[font=footnotesize]{caption}
\usepackage[hidelinks]{hyperref}
\usepackage{thm-restate}
\usepackage{algorithm}
\usepackage{algpseudocode}
\usepackage{dutchcal}

\newtheorem{theorem}{Theorem}[section]
\newtheorem{lemma}[theorem]{Lemma}
\newtheorem{proposition}[theorem]{Proposition}
\newtheorem{corollary}[theorem]{Corollary}
\newtheorem{definition}[theorem]{Definition}
\newtheorem{condition}[theorem]{Condition}
\newtheorem{alg}[theorem]{Algorithm}

\theoremstyle{remark}
\newtheorem{remark}[theorem]{Remark}

\numberwithin{equation}{section}

\DeclareMathOperator{\Spec}{Spec}

\DeclareMathOperator{\vol}{vol}
\newcommand{\norm}[1]{\lVert #1 \rVert}

\newcommand{\Rn}{\mathbb{R}^n}

\newcommand{\spec}{\textrm{spec}}

\begin{document}

\title[Computing spectrum of infinite-volume operators from local patches]{Computing the spectrum and pseudospectrum of infinite-volume operators from local patches}

\author{Paul Hege}
\address{Paul Hege, Department of Neural Dynamics and MEG, Hertie Institute for Clinical Brain Research, Otfried-M\"uller-Str.~25, 72076 T\"ubingen, Germany; Werner Reichardt Centre for Integrative Neuroscience, Otfried-M\"uller-Str.~25, 72076 T\"ubingen, Germany; MEG Center, University of Tübingen, Otfried-M\"uller-Str.~47, 72076 T\"ubingen, Germany}
\email{paul-bernhard.hege@uni-tuebingen.de}

\author{Massimo Moscolari}
\address{Massimo Moscolari, Dipartimento di Matematica, Politecnico di Milano, Piazza Leonardo da Vinci 32, 20133 Milano, Italy}
\email{massimo.moscolari@polimi.it}

\author{Stefan Teufel}
\address{Stefan Teufel, Mathematisches Institut, University of T\"ubingen, Auf der Morgenstelle 10, 72076 T\"ubingen, Germany}
\email{stefan.teufel@uni-tuebingen.de}


\date{15 April 2024, and, in revised form, December 10, 2024}


\begin{abstract}
We show how the spectrum of normal discrete short-range infinite-volume operators can be approximated with two-sided error control using only data from finite-sized local patches.  As a corollary, we prove the computability  of the  spectrum of such  infinite-volume operators with the additional property of finite local complexity and provide an explicit algorithm. Such operators appear in many applications, e.g.\ as discretizations of differential operators  on unbounded domains or as so-called tight-binding Hamiltonians in solid state physics. For a large class of such operators, our result allows for the first time to establish computationally also the absence of spectrum, i.e.\  the existence and the size of spectral gaps.
We   extend our results to the $\varepsilon$-pseudospectrum of non-normal operators, proving that also the pseudospectrum of such operators is computable.
\end{abstract}

\maketitle

\tableofcontents

\section{Introduction}

The computation of spectra of linear operators is a fundamental problem that has been studied at various levels of generality. For finite-dimensional matrices, the theorem of Abel and Ruffini shows that there is no closed-form solution \cite{abel,ruffini}, but the eigenvalues can be approximated numerically \cite{vonmises_1929,golub_2001,saad_2011}. The approximation error can be bounded using statements such as  the Gershgorin and Bauer-Fike theorems \cite{gershgorin,bauer_fike,taussky_1988}, and there are also algorithms that compute validated intervals for the eigenvalues, thus providing rigorous error control \cite{wilkinson_1961,yamamoto_1980,yamamoto_1982}.

Operators on infinite-dimensional spaces are usually studied by restriction to a finite-dimensional subspace. Rigorous upper bounds on the eigenvalues can be derived from finite-dimensional approximations using the Rayleigh-Ritz method \cite{ritz_1909,plum_2021}. A number of methods have been proposed over the years to compute complementary lower bounds on the eigenvalues \cite{weinstein_1937,kato_1949,weyl_1950,bazley_1961,bazleyFox1961,bazleyFox1966,behnke_1994,behnke_1995,daviesHierarchical}. These methods usually assume that the spectrum  in a certain energy window
consists of a finite number of eigenvalues \cite{plum_2021}. Thus, they are applicable for example to  differential operators on compact domains. For this setting, e.g.\ finite element methods with error control are   available \cite{carstensen_2014,carstensen_2014a,liu_2015,hu_2016}.

In this paper, we consider the problem of computing spectra with rigorous and explicit  error control for \textit{short-range infinite volume operators}, a class of operators for which the known methods do not apply and which are at the same time very relevant for applications.
We say that a bounded operator $H$ on a separable Hilbert space $\mathcal H$
is a short-range infinite volume operator, if
 there is an  orthonormal basis $(e_x)_{x \in \Gamma}$ indexed by a uniformly discrete subset $\Gamma \subset \Rn$ for some $n\in\mathbb{N}$, such that the matrix elements $H_{xy}:= \langle e_x,He_y\rangle$ of $H$ fulfill the \textit{short-range} condition
\begin{align*}
|H_{xy}| \leq C \frac{1}{d(x,y)^{n+\varepsilon}}
\end{align*}
for some $C$, $\varepsilon > 0$ and all $x,y\in\Gamma$. Here $d$ denotes the maximum distance on $\mathbb{R}^n$, cf.\ \eqref{DefMaxDist}. Important examples of such short-range infinite volume operators are discretizations of differential operators on infinite domains and so-called tight-binding models from solid state physics.
In Section~\ref{subsec-examples} we briefly discuss examples of operators arising from applications for which computability of the spectrum was not previously known and to which our algorithm can be applied.

Since there is an extensive mathematical literature dealing with the spectral problem for infinite volume operators, we will briefly comment on some recent results.  The easiest method to implement is probably the finite section method, which is not only widely used in practice, but has also been studied theoretically from various perspectives \cite{arveson_1993,boettcher_1994,chandlerwilde_2016,lindner_2003,lindner_2006,chandlerwilde_2013}. The finite section $H_\Lambda$ of $H$ on a finite subset $\Lambda\subset\Gamma$ is just the square matrix $(H_{xy})_{x,y\in\Lambda}$. For certain classes of operators, sequences of finite sections can be found such that their spectra are provably convergent. However, the occurrence of spectral pollution at the boundaries makes it difficult to achieve good error control for general operators using finite sections \cite{daviesPollution,lewinPollution}.

Another popular approach to computing the spectrum of infinite-volume operators with strong error control is based on approximating aperiodic operators by periodic ones, e.g.\ \cite{beckusTakase,beckusDelone,colbrook_2020}. Bounding the approximation error of such periodic approximations has motivated a number of results proving the Lipschitz or Hölder continuity of the spectrum in the coefficients for a large   class of operators connected to dynamical systems \cite{bellissardOneDimensional,discreteGroups}. Because there are easily computable bounds on the Lipschitz constants, these continuity results can be used to provide error control for the convergence of the approximant spectra in Hausdorff distance.

To make practical use of these bounds, however, suitable periodic approximations are required. To use the Lipschitz continuity of the spectrum, the periodic approximant has to have the same set of local patches as the infinite-volume operator at a certain scale. While the existence of such periodic approximations for substitution systems is increasingly well-understood in one dimension \cite{beckus_2018a,damanik2022one,tenenbaum_2024} the situation in higher dimensions is more complex and under active investigation \cite{beckus2021symbolic,band_2024}. In both cases, there are important examples of operators that cannot be approximated periodically, such as the jump potential in one dimension, or the two-dimensional Penrose tiling, for which the local matching rules force an aperiodic pattern \cite{penrose_1979}. Therefore, while the dynamical systems method provides strong error control, it requires the construction of periodic approximants, which is not always possible and for which no general algorithm is known.

A different way to compute spectra is the recently proposed method of \textit{uneven sections} \cite{colbrook_2019,colbrook_2020}, in which the operator $H$ is approximated by a rectangular matrix $(H_{xy})_{x\in\Lambda',y\in\Lambda}$ for some finite $\Lambda\subset\Lambda'\subset\Gamma$.  In addition to effectively reducing spectral pollution compared to the finite section method, this method also provides one-sided error control \cite{colbrook_2019}.

A central object in the following discussions is  the so-called \textit{lower norm function}
\begin{align}
\label{eqRho}
 \rho_H(\lambda) = \begin{cases} \big\lVert (H - \lambda)^{-1} \big\rVert^{-1} &\text{for }\lambda \notin \Spec(H) \\
          0  &\text{otherwise}\,.
           \end{cases}
\end{align}
For normal operators $H$ it satisfies
$\rho_\lambda(H) = d(\lambda, \Spec(H))$ and  for general operators it can serve as a defintion of the  $\varepsilon$-pseudospectrum, cf.\ \eqref{DefPseudo}.

The method from \cite{colbrook_2019} implies (cf.\ Theorem~\ref{upperBoundTheorem} below) that for every  patch size $L > 0$  and center point $x\in \mathbb{R}^n$  the smallest singular value  $\varepsilon_{L, \lambda, x}\geq 0$ of the rectangular matrix  $(H_{yz})_{y\in B_{L+m}(x),z\in B_{L}(x)}$ satisfies
\begin{align}
\label{colbrookUpperBound}
\rho_H(\lambda) \leq \varepsilon_{L, \lambda, x}\,
\end{align}
and thus, for normal $H$, also
\begin{align}
\label{colbrookUpperBound2}
d(\lambda, \Spec(H)) \leq \varepsilon_{L, \lambda, x}\,.
\end{align}
Here $B_L(x)\subset \mathbb{R}^n$ denotes the hypercube around $x$ with side length $2L$ and $m$ is a fixed finite number.

This form of error control is only one-sided, however, as there is no lower bound for $\rho_H(\lambda)$. In fact, the authors of \cite{colbrook_2019} prove that no algorithm giving a lower bound on $\rho_H(\lambda)$ can exist as long as the operator is given solely by its matrix elements \cite{hansen_2011}.
The impossibility of a lower bound implies that only the existence of spectrum in a certain interval can be shown, while the absence of spectrum (a spectral gap) can not be rigorously established in this general setting. This also implies that it is not possible to give a bound on the Hausdorff distance to the infinite-volume spectrum, as for example in the periodic approximation approach of \cite{beckusTakase,beckusDelone}.

However, for many applications involving infinite volume operators the existence and size of spectral gaps is of central improtance.
In a previous paper, we have shown that the no-go theorem which rules out a lower bound on $\rho_H(\lambda)$ can be circumvented in most cases of physical interest, and have given a practical algorithm to compute validated spectral gaps of an infinite-volume system by considering all subsystems of a given size \cite{hege_2022}. In the present work, we expand on this by giving a general algorithm to compute the spectrum of operators with finite local complexity (\textit{flc}) with full error control.
The primary insight making this possible is that for finite range operators knowledge of the infimum of
$\varepsilon_{L, \lambda, x}$ over \emph{all centers} $x\in\Rn$ is sufficient to provide a quantitative lower bound  on $\rho_H(\lambda)$. More precisely, for
\begin{align*}
\varepsilon_{L, \lambda} := \inf_{x \in \Rn} \varepsilon_{L, \lambda, x}\,,
\end{align*}
we prove the following theorem.

\begin{theorem}
\label{introThmNew}
Let  $n \in \mathbb{Z}_+$, let  $\Gamma \subseteq \Rn$ be a uniformly discrete   set with  packing radius $q>0$, and let $H$ be a bounded operator on $\ell^2(\Gamma)$ with finite range $m$ and $M:= \sup_{x,y\in\Gamma}|H_{xy}|$.
Then for every $\lambda\in\mathbb{C}$ and every $L>m$ it holds that
\begin{align}
\rho_H(\lambda) \;\geq\; \varepsilon_{L, \lambda} - \tfrac{C}{L} \,,
\label{ineqIntroThmNew}
\end{align}
where
\begin{align}
 \label{definitionC}
C := m\,M   \left( \frac{36m}{q} \right)^{n/2}\,.
\end{align}
\end{theorem}

The practical significance of Theorem~\ref{introThmNew} lies in the observation that for \textit{flc} operators, the infimum $\varepsilon_{L, \lambda}$ can be computed by evaluating $\varepsilon_{L,\lambda,x}$ for a finite number of suitably chosen centres $x \in \mathbb{R}^n$. We show that combining \eqref{colbrookUpperBound} and \eqref{ineqIntroThmNew} then leads to a general algorithm for computing the spectrum of \textit{flc} operators with error control. That is, given any short-range, discrete, normal, \textit{flc} operator $H$ and $k \in \mathbb{N}$, our algorithm computes an approximation $\Gamma_k(H) \subseteq \mathbb{C}$ such that
\begin{align}
\label{eqIntroHausdorffConv}
d_\mathrm{H}(\Spec(H), \Gamma_k(H)) \leq 2^{-k}\,,
\end{align}
where $d_\mathrm{H}$ is the Hausdorff distance on subsets of $\mathbb{C}$. Thus the algorithm is able to compute the spectrum of \textit{flc} infinite-volume operators to any given precision in Hausdorff distance.

Inequalities similar to \eqref{ineqIntroThmNew}  have been established for one-dimensional systems in \cite{chonchaiya_2011,chandlerwilde_2013,chandlerwilde_2023,chandlerwilde_2024,lindner_unpublished}, where $\Gamma=\mathbb{Z}\subset\mathbb{R}$,   $H$ is a band matrix (or band-dominated matrix) acting on $\ell^2(\mathbb Z, X)$. In contrast to our work, which focuses on the Hilbert space setting, these papers have studied the situation in general Banach spaces $X$ (in which case some additional assumptions on $H^* - \lambda \mathbf{1}$ are needed).

In an earlier version of this paper,  we used methods similar to~\cite{hege_2022} to obtain a lower bound on~$\rho_H(\lambda)$. However, while the resolvent method described in~\cite{hege_2022} can be more efficient for concrete computations by exploiting the exponential decay of edge states with distance to the edge, applying these methods in the current context leads to a more complicated lower bound on $\rho_H(\lambda)$ with $1/\sqrt{L}$ asymptotics.  In this revised version we have instead taken inspiration from the methods in~\cite{lindner_unpublished} and prove a lower bound on $\rho_H(\lambda)$ with $1/L$ asymptotics. The result for one-dimensional systems shown in \cite{chandlerwilde_2024} suggests that $1/L$ is the optimal asymptotic.

For non-normal operators, we can use Theorem~\ref{introThmNew} to show the computability of the $\varepsilon$-pseudospectrum \cite{landau_1975,varah_1979,hinrichsen_1991,trefethenEmbreeBook} $\Spec_\varepsilon(H)$. Again, we provide an algorithm which computes for any $\varepsilon > 0$ an approximation $\Gamma_k(H, \varepsilon)$ such that
\begin{align}
\label{eqIntroHausdorffConvPseudo}
d_H(\Spec_\varepsilon(H), \Gamma_k(H, \varepsilon)) \leq 2^{-k}\,.
\end{align}

The existence of algorithms fulfilling \eqref{eqIntroHausdorffConv} and \eqref{eqIntroHausdorffConvPseudo} should be viewed in the context of previous results about the computability of the general infinite-dimensional spectral problem \cite{hansen_2011}. In extending the theory of computability from discrete computations to analytical and numerical problems \cite{turing_1937, blum_1989, blum_1998}, the \textit{solvability complexity index} (SCI) is a very useful classification of computational problems by the number of limits required for their solution \cite{hansen_2011}. The spectral problem for operators on infinite-dimensional spaces has been a particular focus of investigation for determining the SCI \cite{daviesLectureNote,hansen_2008,hansen_2010,colbrook_2022,hansen_2016,universalPeriodic,geometricFeatures,roeslerSCI}.

If a general infinite-volume operator is given by its matrix entries, it has been shown that it is impossible to compute the spectrum with error control \cite{colbrook_2019}. But it is possible to compute the spectrum via a series of convergent estimates (without error control), which places the general spectral problem for self-adjoint operators in the ${\rm SCI} = 1$ class, which require one limit to solve \cite{hansen_2011}. The precise SCI has also been determined for many spectral problems differing in the conditions placed on $H$ \cite{benartzi_2020}. The possibility of one-sided error control by an upper bound on the distance to spectrum has also inspired an intermediate SCI class $\Sigma_1$ that is in between the classes $\Delta_1$ and $\Delta_2$, the classes of SCI $= 1$ with and without error control, respectively \cite{colbrook_2022,benartzi_2020}.

The authors of this no-go theorem have stressed, however, that the SCI of a given computation can be lowered through additional structure and conditions on the general problem \cite{benartzi_2020}. This is precisely what we achieve here by introducing the structure of finite local complexity. It turns out that by adding this requirement, the spectrum becomes computable and the \textit{flc} spectral problem is therefore in class $\Delta_1$~(full error control), instead of the class $\Delta_2$~(SCI~$= 1$) or $\Sigma_1$ (one-sided error control), which contains the general spectral problem. This lowering of the SCI is especially interesting in light of the fact the \textit{flc} is a rather general condition that is fulfilled by most operators that occur in practice.

\section{Definitions of computational problems}
\label{sec-definitions}

To prove the computability of the spectrum for operators of finite local complexity, we require a clear notion of computational problems and their solvability. Following \cite{benartzi_2020}, we define a computational problem by the following data:

\begin{definition}
 A \textit{computational problem} is a tuple $(\Omega, \Lambda, (\mathcal M, d), \Xi)$, where
 \begin{itemize}
  \item $\Omega$ is a set (the ``set of problems'');
  \item $\Lambda$ is a family of functions $\Lambda = (f_i)_{i\in\mathcal{I}}$, indexed by a countable set $\mathcal{I}$, where each $f_i$ is a function $f_i : \Omega \to \mathbb{R}$ (the ``evaluation functions'');
  \item $(\mathcal M, d)$ is a metric space (the metric space of ``possible solutions'');
  \item $\Xi$ is a function $\Omega \to \mathcal M$ (the ``problem function'').
 \end{itemize}
\end{definition}

\noindent
In this definition, the set $\Omega$ is the set of concrete problems that an algorithm has to be able to solve. For spectral problems, this would correspond to a certain set of operators. The exact solution of the problem is given by the function $\Xi$, which takes values in the metric space $(\mathcal M, d)$. For the spectral case, we would choose $\Xi(A) = \Spec(A)$ for any $A \in \Omega$, and $\mathcal{M}$ would be the power set of $\mathbb{C}$. Because we approximate the spectrum and pseudospectrum in Hausdorff distance, we equip $\mathcal M$ with the Hausdorff distance $d_H$. The functions $f_i \in \Lambda$, finally, define how the algorithm can get information about the concrete problem. For the spectral problem as defined in \cite{colbrook_2019}, for example, the evaluation functions would return the matrix elements of the operator in a certain basis, for example. We would like to stress that the choice of evaluation functions $f_i \in \Lambda$ can be decisive for the solvability or unsolvability of a computational problem.

Computational problems can be solved by algorithms. The solvability of course depends on what kinds of computations one allows the algorithm to perform (for example, whether only algebraic or more general computations are allowed) \cite{benartzi_2020}. Unless otherwise noted, in the following we will understand computability to refer to algorithms that can be executed by BSS machines \cite{blum_1989,blum_1998}, a classical framework for computations with real numbers. In Appendix~\ref{appendixBSS}, we describe our model of computation in more detail, including what it means for the BSS algorithm to sequentially access information about the computational problem via the evaluation functions $f_i \in \Lambda$.

In this paper, we consider the spectral and pseudospectral problem with the additional structure of finite local complexity (\textit{flc}). We can thus circumvent the impossibility result of \cite{colbrook_2019} by using a restricted set $\Omega$. But the resulting class of \textit{flc} operators is still very general, and can accomodate many if not most physical situations, including all examples from \cite{colbrook_2019}. We show that the spectral problem becomes computable when considering operators of finite local complexity.
In addition to restricting the problem set $\Omega$, the set of evaluation functions $\Lambda$ must be augmented in order to allow the algorithm to make use of the \textit{flc} structure.

\begin{definition}
\label{definitionUniformlyDiscrete}
 A subset $\Gamma \subseteq \Rn$ is called \textit{uniformly discrete} if there exists a constant $q > 0$, the \textit{packing radius}, such that $d(x,y)  \geq q$ for all $x, y \in \Gamma$ with $x\neq y$.
\end{definition}

In the above definition, and for the rest of this paper, we define the distance $d(x,y)$ on $\Rn$ as the {\em maximum distance}
\begin{align}\label{DefMaxDist}
d((x_1, \dots, x_n), (y_1, \dots, y_n)) = \max_{k = 1, \dots, n} |x_k - y_k|\,.
\end{align}
 We also define $B_r(x)$, for $x \in \Rn$ and $r > 0$, as the open ball using this distance; that is, the set $B_r(x)$ is a hypercube with side length $2r$ centered at $x$.

\begin{definition}
\label{definition-discrete-operator}
We define a \textit{discrete operator} $H$ in dimension $n\in\mathbb{N}$ as a bounded operator on a separable Hilbert space $\mathcal H$, together with an orthonormal basis $(e_i)_{i \in \Gamma}$ indexed by a uniformly discrete subset $\Gamma \subset \Rn$. We define the \textit{matrix elements} at points $x, y \in \Gamma$ as $H_{xy} = \langle e_x, H e_y\rangle$.
\end{definition}

In the following, we will always represent discrete operators $H$ with respect to the special basis  $(e_i)_{i \in \Gamma}$
and use the basis isomorphism to identify $\mathcal{H}$ with $\ell^2(\Gamma)$. Furthermore, for any $x\in\mathbb{R}^n$ and $L>0$, the finite dimensional subspace $\mathcal{H}_{B_L(x)}\subset \mathcal{H}$
is defined by $\mathcal{H}_{B_L(x)} := \mathrm{span}\{ e_x\,|\, x\in B_L(x)\}$, and the orthogonal projection onto $\mathcal{H}_{B_L(x)}$ is denoted by $\mathbf{1}_{B_L(x)}$.
\smallskip

We now define our two main conditions on $H$, short-range and finite local complexity.

\begin{definition}
\label{definition-short-range}
Let $H$ be a discrete operator in dimension $n$. Then $H$ is called \textit{short-range} if there exist $C$ and $\varepsilon>0$ such that
\begin{align*}
 |H_{xy}| \leq C \,d(x,y)^{- (n+\varepsilon)}\,.
\end{align*}
for all $x, y \in \Gamma$. $H$ is said to have \textit{finite range} if there is a number $m > 0$ (the maximal hopping length) such that $H_{xy} = 0$ whenever $d(x,y) > m$.
\end{definition}

The condition of finite local complexity is usually defined for point sets \cite{lagarias_1996,lagarias_1999a,lagarias_1999,besbesFLC}. Very succintly, a uniformly discrete set $\Gamma \subset \mathbb{R}^n$ is defined to have \textit{flc} iff $\Gamma - \Gamma$ is discrete. This turns out to be equivalent to the set of finite patches $\{\Gamma \cap B_L(x)\,|\,x \in \Gamma\}$, falling into finitely many equivalence classes under translation \cite{besbesFLC} for each $L>0$.  To extend this concept to operators, we require that there are finitely many equivalence classes on which, additionally, the operator $H$ acts in the same way, which we define precisely as follows.

\begin{definition}
\label{def:EquivalentAction}
 A discrete operator $H$ is said to have \textit{equivalent action} on two subsets $A, B \subseteq \Gamma$ if there is a $t \in \Rn$ such that $B = t + A$ and if there exist $U(z) \in S^1 \subseteq \mathbb C$ for every $z \in A$ such that for any $a_1, a_2 \in A$ we have
 \begin{align*}
  H_{b_1 b_2} = U(a_1) H_{a_1 a_2} U(a_2)^*\,,
 \end{align*}
 where $b_1 = a_1 + t, b_2 = a_2 + t$.
\end{definition}

\begin{remark}
It is clear that for any operator $H$, equivalent action of $H$ defines an equivalence relation on subsets of $\Gamma$. The complex phases $U(a_1)$ and $U(a_2)$ can often be set to unity, but they are necessary as gauge transformations for certain operators, in particular for discrete Schrödinger operators with magnetic fields.
\end{remark}

\begin{definition}
\label{definition-finite-local-complexity}
 A discrete operator $H$ is said to have \textit{finite local complexity} if
 for any $L > 0$, the set $\{ \,\Gamma \cap B_L(x) \,|\, x \in \Rn \,\}$ is contained in finitely many equivalence classes with respect to equivalent action of $H$.
\end{definition}

\begin{remark}
More explicitly, $H$ has finite local complexity if for any $L>0$ there are finitely many $x_1, \dots, x_s \in \Rn$ such that for any $y \in \Rn$, there is a $k \in \{1, \dots, s\}$ such that $H$ has equivalent action on $\Gamma \cap B_L(y)$ and $\Gamma \cap B_L(x_k)$.
\end{remark}

We now define the spectral and pseudospectral problems for operators of finite local complexity. The main goal of this paper is to show that these problems are solvable with error control. We will  consider the spectral problem only for normal operators. The spectral problem for non-normal operators is intractable because the spectrum of non-normal operators is not Hausdorff-continuous in the matrix entries, even for \textit{flc} operators \cite{trefethenEmbreeBook}. Instead, we show the that the $\varepsilon$-pseudospectrum is computable for non-normal operators, for all $\varepsilon > 0$.

To define the spectral problem for operators of finite local complexity, some work has to be done to define suitable evaluation functions which will allow the algorithm to make use of the \textit{flc} structure. There are multiple ways to do this; for example, it would be sufficient to give the algorithm access to the \textit{repetitivity function}, which gives a radius in which all patches of size $L$ occur \cite{besbesFLC}. However, to stay closer to the way actual implementations are likely to operate, we here define the evaluation so as to allow the algorithm to enumerate all local patches of a given size and to access the matrix elements and site locations for any given patch. It is expected that for concrete applications, special-purpose algorithms for the enumeration of local patches will be used. We have demonstrated this for the case of cut-and-project quasicrystals in \cite{hege_2022,hegeThesis}. Because the exact definition of the evaluation functions is a bit tedious, it is given in Appendix~\ref{appendixEvaluationFunctions}. We now define the \textit{flc} version of the spectral computational problem as follows.

\begin{definition}
\label{definitionSpectralProblem}
 The \emph{flc spectral problem} is the computational problem   $( \Omega, \Lambda, (\mathcal M, d_H), \Xi)$, where
 \begin{itemize}
  \item $\Omega$ is the set of normal discrete operators with finite local complexity and short-range;
  \item $\Lambda$ is a family of functions $(f_i)_{i \in \mathcal I}$ satisfying     Conditions~\ref{definitionEvaluationFunctions} in Appendix~\ref{appendixEvaluationFunctions};
  \item $\mathcal M$ comprises all compact subsets of $\mathbb{C}$, and $d_H$ is the Hausdorff distance;
  \item $\Xi$ is the function which assigns to every operator $H$ its spectrum, $\Xi(H) = \Spec(H)$.
 \end{itemize}
\end{definition}

\noindent
One of the main results in this paper is that the \textit{flc} spectral problem is solvable with error control in Hausdorff distance.

\begin{theorem}
\label{theoremComputabilitySpectrum}
Let $(\Omega, \Lambda, (\mathcal M, d_H), \Xi)$ be the \emph{flc} spectral problem. Then for every $k \in \mathbb{N}$ there exists a Blum-Shub-Smale (BSS) algorithm $\Gamma_k : \Omega \to \mathcal M$, using the family of evaluation functions~$\Lambda$, such that
\begin{align*}
d_H(\Gamma_k(H), \Xi(H)) \leq 2^{-k}
\end{align*}
for all $H \in \Omega$.
\end{theorem}

\noindent
The proof of this Theorem is given in Sections~\ref{sec-computability-lower-norm} and~\ref{sec-computability-spectrum}.

\subsection{The pseudospectrum}\label{SecPseudo}

For non-normal operators, we cannot compute the spectrum itself, but it is still possible to compute the $\varepsilon$-pseudospectrum for $\varepsilon > 0$. The pseudospectrum may be defined in terms of the \textit{lower norm function} $\rho_H : \mathbb{C} \to \mathbb{R}_+$ \cite{coreIssue}, which is defined as
\begin{align}
 \rho_H(\lambda) = \begin{cases} \big\lVert (H - \lambda)^{-1} \big\rVert^{-1} &\text{for }\lambda \notin \Spec(H) \\
          0  &\text{otherwise}\,.
           \end{cases}
\end{align}
The lower norm function is not a norm, but it can be given a variational definition that is similar to the operator norm:
\begin{align}
\label{eqRhoVariational}
\rho_H(\lambda) = \inf_{\psi \in \mathcal{H}\setminus\{0\}} \frac{\norm{(H-\lambda) \psi}}{\norm{\psi}}\,.
\end{align}
Moreover, $\rho_H$ is Lipschitz continuous with Lipschitz constant~$1$ (see, for example, \cite{noteOnHausdorff}, Lemma 2.1).
The $\varepsilon$-pseudospectrum is defined as the closed $\varepsilon$-sublevel set of~$\rho_H$:
\begin{align}\label{DefPseudo}
\Spec_\varepsilon(H) := \{ z \in \mathbb{C} \;|\; \rho_H(z) \leq \varepsilon \}\,.
\end{align}
Some authors also define the pseudospectrum as the open set $\Spec_\varepsilon^\circ(H) = \{ z \in \mathbb{C} \;|\; \rho_H(z) < \varepsilon \}$, using a strict inequality. This has no influence on the computability, however, because $\Spec_\varepsilon(H)$ is the closure of $\Spec_\varepsilon^\circ(H)$ \cite{trefethenEmbreeBook}, and thus $\Spec_\varepsilon(H)$ and $\Spec_\varepsilon^\circ(H)$ have Hausdorff distance zero. Consequently, any algorithm computing $\Spec_\varepsilon(H)$ with error control, according to Equation \eqref{eqIntroHausdorffConvPseudo}, also computes $\Spec_\varepsilon^\circ(H)$ in this sense and vice versa.

For normal operators, we have
\begin{align}
\label{rhoForNormal}
\rho_H(\lambda) = d(\lambda, \Spec(H))\,,
\end{align}
thus the $\varepsilon$-pseudospectrum of a normal operator $H$ is just the $\varepsilon$-fattenig of its spectrum. For non-normal operators, we still have
\begin{align*}
\rho_H(\lambda) \leq d(\lambda, \Spec(H))\,.
\end{align*}
Thus the $\varepsilon$-pseudospectrum always contains the $\varepsilon$-fattening of $\Spec(H)$. A computational problem for the $\varepsilon$-pseudospectrum can be defined similarly as for the spectrum:

\newcommand{\ps}{{\rm ps}}
\begin{definition}
\label{definitionPseudospectralProblem}
 The \emph{flc pseudospectral problem} is the computational problem  $(\Omega_\ps, \Lambda_\ps, (\mathcal M, d_H), \Xi_\ps)$, where
 \begin{itemize}
  \item $\Omega_\ps$ is the set of all tuples $(H, \varepsilon)$, where $H$ is a discrete operator with finite local complexity and short-range (but not necessarily normal), and $\varepsilon > 0$;
  \item $\Lambda_\ps$ consists of the same evaulation functions of Definition~\ref{definitionSpectralProblem}, applied to $H$, and one additional evaluation function which returns $\varepsilon$;
  \item $(\mathcal M, d_H)$ is as in Definition~\ref{definitionSpectralProblem};
  \item $\Xi_\ps(H, \varepsilon) = \Spec_\varepsilon(H)$.
 \end{itemize}
\end{definition}

The following theorem states that the $\varepsilon$-pseudospectrum, for $\varepsilon > 0$, is computable with error control, just like the spectrum for normal operators.

\begin{theorem}
 \label{theoremComputabilityPseudospectrum}
Let $(\Omega_\ps, \Lambda_\ps, (\mathcal M, d), \Xi_\ps)$ be the flc pseudospectral problem. Then for every $k \in \mathbb{N}$ there exists a Blum-Shub-Smale (BSS) algorithm $\Gamma_k : \Omega_\ps \to \mathcal M$, using the family of evaluation functions $\Lambda_{ps}$, such that
\begin{align*}
 d(\Gamma_k(H, \varepsilon), \Xi_\ps(H, \varepsilon)) \leq 2^{-k}
\end{align*}
for all $H \in \Omega_\ps$.
\end{theorem}

\noindent
The proof of Theorem~\ref{theoremComputabilityPseudospectrum} is given in Section~\ref{sec-computability-pseudospectrum}.

\begin{figure}
\begin{center}
\includegraphics[width=0.84\linewidth]{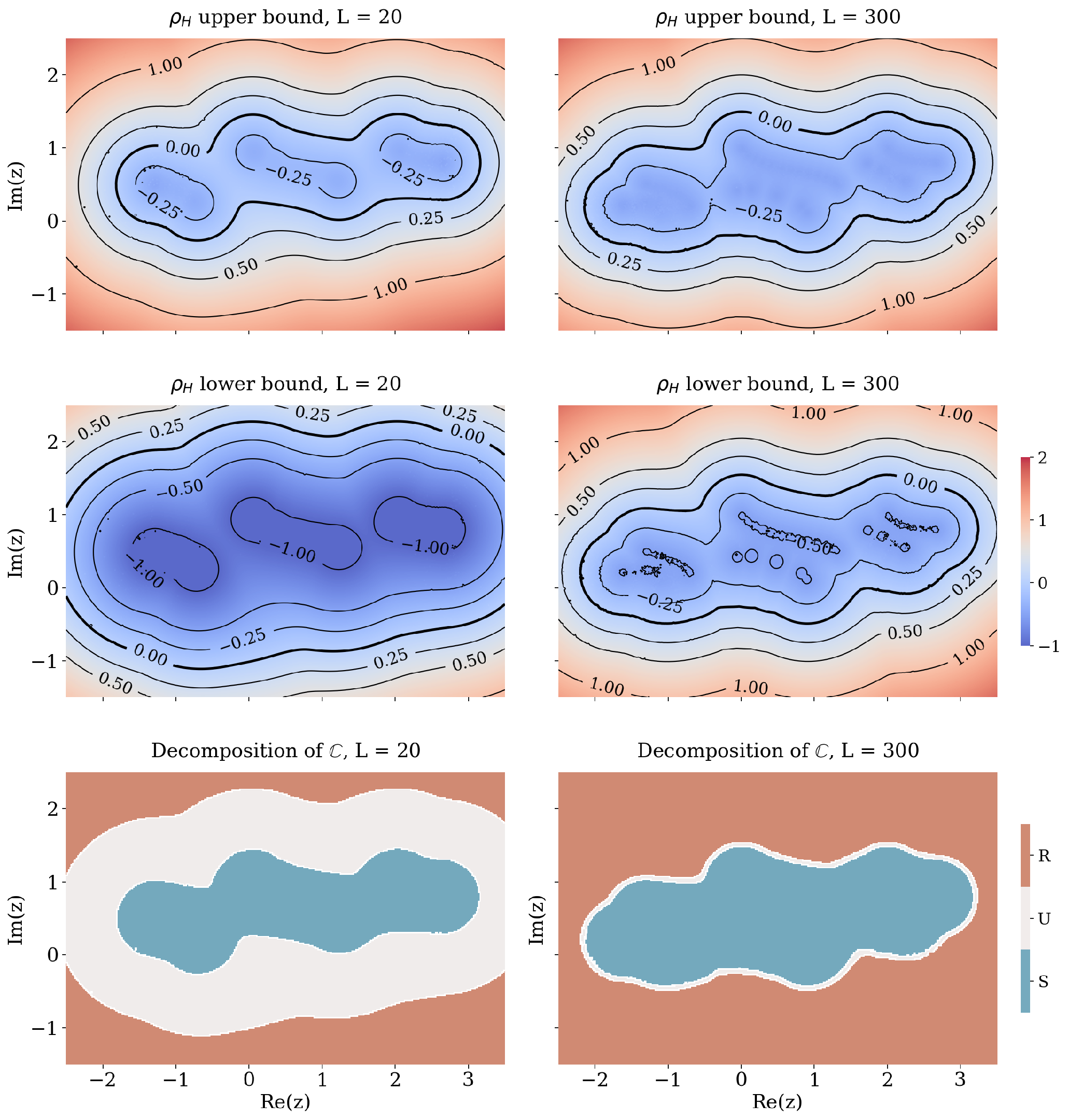}
\end{center}
\captionsetup{width=0.8\linewidth}
\caption{Exact computation of the $\varepsilon$-pseudospectrum of a non-Hermitian Hamiltonian with a cut-and-project potential for $\varepsilon = 0.5$. The Hamiltonian is defined by $H\psi(n) = -\psi(n-1) + V(n)\psi(n) - \psi(n+1)$. The chosen potential has the form $V(n) = (1 + i)\mathbf{1}(\alpha n < 1 / \alpha)$, where we chose $\alpha = 1.66$ (For $\alpha = (1 + \sqrt 5) / 2$, this construction gives the Fibonacci quasicrystal, but we chose $\alpha = 1.66$ because it creates a less uniform potential leading to a slower convergence and thus a more pronounced effect of increasing $L$ in the pictures.). The left and right column show the same computation with $L = 20$ and $L = 300$, respectively. The uppermost row shows the lower, spectral gap bound on $\rho_H$, while the row below shows the upper bound on $\rho_H$ from \cite{colbrook_2019}. If the lower bound is positive, the associated point is known   to be in the complement of  the pseudospectrum; if the upper bound is negative, the point is known to be inside the pseudospectrum. This gives a decomposition of the plane into three sets $R, U,$ and $S$ of points, where it is known that $S \subseteq \Spec_\varepsilon(H)$, that $R \cap \Spec_\varepsilon(H) = \varnothing$, and no statement can be made about $U$. This is very similar to the sets $S_\tau$, $R_\tau$ and $U_\tau$ in Section~\ref{sec-computability-pseudospectrum}, except that here we have fixed $L$ instead of $\tau$ and we do not vary the spacing of the grid on which $\tilde\rho_H(\lambda, \tau)$ is evaluated.\label{figUpperLowerBounds}}
\end{figure}

\begin{remark}
While this paper focusses on a theoretical computability result, the described algorithms also constitute a practical method for computating the (pseudo-) spectrum in systems of finite local complexity. In \cite{hege_2022}, we have already described an algorithm to prove the existence of spectral gaps using an edge state criterion on finite patches. For this we used Dirichlet boundary conditions and computed an edge state criterion to provably avoid spectral pollution. While this method could be used to prove the computability of the spectrum \cite{hegeThesis}, here we instead use the method of uneven sections \cite{colbrook_2019}, which enables us to give a unified treatment of the computation of the spectrum and pseudospectrum. Although it is not as efficient as the Dirichlet-based method of \cite{hege_2022} for higher-dimensional system, we can use this approach to compute gap bounds for the $\varepsilon$-pseudospectrum. Figure~\ref{figUpperLowerBounds} shows a computation of the spectral inclusion and spectral gap bounds for the pseudospectrum of a non-Hermitian Hamiltonian with Fibonacci-like potential \cite{finiteSectionsPeriodic,damanikFibonacci,jagannathan_2021}.
\end{remark}

Our proof of Theorems~\ref{theoremComputabilitySpectrum} and~\ref{theoremComputabilityPseudospectrum} is based on the spectral detectability result, Theorem~\ref{introThmNew}. We have formulated Theorem~\ref{introThmNew} for finite-range as opposed to short-range operators because this simplifies the statement considerably.
In Section~\ref{sec-finite-range}, we show that the computational spectral and pseudospectral problems for short-range  \textit{flc} operators can be reduced to that for finite range operators.
In Section~\ref{sec-upper-bound}, we recall the upper bound $\varepsilon_{L, \lambda, x}$ on $\rho_H(\lambda)$ from \cite{colbrook_2019}, adapted to our setting and finite range interactions. In Section~\ref{sec-quasimode-detectability}, we prove the corresponding lower bound, which is given by Theorem~\ref{introThmNew}. In Section~\ref{sec-computability-lower-norm}, we apply Theorem~\ref{introThmNew} to show that $\rho_H(\lambda)$ is computable at every $\lambda \in \mathbb{C}$. In Section~\ref{sec-computability-spectrum}, we show how to use the computability of $\rho_H(\lambda)$ to show the computability of the spectrum (Theorem~\ref{theoremComputabilitySpectrum}). The computability of the $\varepsilon$-pseudospectrum (Theorem~\ref{theoremComputabilityPseudospectrum}) is shown in Section~\ref{sec-computability-pseudospectrum}.

\subsection{Examples}
\label{subsec-examples}

Infinite-volume operators which are discrete and of finite local complexity, as described in Definitions~\ref{definition-discrete-operator}, \ref{definition-short-range}, and \ref{definition-finite-local-complexity}, occur frequently in physics.
For example, a so-called discrete Schrödinger operator on $\ell^2(\mathbb{Z}^n)$ \cite{katsura_1950,thouless_1974,levi_1979} is of the form
\begin{align*}
(H\psi)(x) = (2n + V(x))\psi(x) \;- \sum_{\substack{y \in \mathbb{Z}^n \\ d(x,y) = 1}} \psi(y)
\end{align*}
and can be obtained as the discretization of a Schrödinger operator on $\mathbb{R}^n$. Here the function $V:\mathbb{Z}^n\to \mathbb{R}$ is called the potential. Despite its simple structure, the spectral problem for  discrete Schrödinger operators is an area of active research \cite{embree2019spectra,rabinovich2006essential,liu2024fermi}.

The spectrum of a discrete Schrödinger operator with  periodic potential $V(x)$ can be computed  using the Bloch-Floquet transform. In this case, the spectrum consists of finitely many intervals  separated by so called  gaps. A more general class of operators is obtained when $V(x)$ is defined via a substitution rule. Such operators are of high interest in mathematical physics due to their connection to quasicrystals. For example, $V(x)$ may be the Fibonacci potential \cite{kohmoto_1983,damanikFibonacci,jagannathan_2021} or the Thue-Morse sequence \cite{bellissard1990spectral}. Important mathematical questions about these operators concern the nature and fractal dimension of their spectrum \cite{suto_1987,damanik2008fractal} and the gap labelling \cite{bellissard_1992,kellendonk_1994}.

Despite many results on the fractal nature of the spectrum, for numerical investigation the spectrum is usually approximated by reverting to periodic approximants \cite{damanik2015spectral}. While good error bounds on the spectrum of such approximations exist \cite{beckus_2018,beckusTakase} and can be computed in concrete cases, their application requires specialized analysis to find suitable periodic approximations and compute the constants in the bound \cite{beckus_personal_communication}. In contrast, our method only requires the enumeration of local patches of a given size, and therefore provides the first immediately applicable algorithm to compute the spectrum of discrete Schrödinger operators with potentials given by substitution rules.

The Schrödinger operator can also be discretized when a magnetic field is present. In the case of a constant magnetic field in a two-dimensional square lattice, this leads to the Hofstadter model \cite{Hofstadter1976,satija_2016}, which is famous for the self-similar, fractal structure of its spectrum as a function of the magnetic field (Hofstadter butterfly). The spectrum of the Hofstater model is usually approximated using periodic approximations; however, this is only possible for rational fluxes \cite{Hofstadter1976}. For general fluxes, our method is to our knowledge the first that allows computing the spectrum with error control to arbitrary precision in the infinite Hofstadter model.

The Hofstadter model has also been investigated on quasiperiodic lattices like Ammann-Beenker and Rauzy tilings, which lead to a butterfly structure that depends on the quasiperiodicity \cite{kohmoto_1983,fuchs_vidal_butterfly,jagannathan_2014,fulga_aperiodic_2016,fuchs_2018}. Our method is the first to allow computing the spectrum of such systems with error control, which we have also implemented for concrete examples in  \cite{hege_2022}.

A different generalization is to consider  Schrödinger operators on arbitrary graphs \cite{sy1992discrete,korotyaev2014schrodinger}. Because we do not presuppose that our operators are defined on a grid, or the existence of a groupoid structure as in \cite{beckusTakase}, our algorithm can be readily applied to any Schrödinger operator on arbitrary geometrically embedded infinite graphs as long as their local structure is known.

\subsection{Random operators}

Our method can also be applied to random operators, which play an important role in physics as models of noise and impurities \cite{anderson_1958,halperin_1967,aizenman_warzel}. For example, choosing $\Gamma = \mathbb{Z}^n$, a random operator may be defined as a weakly measurable function defined on a probability space $(\Omega, \mathcal A, \mathbb{P})$ which assigns to every $\omega \in \Omega$ a bounded operator $H_\omega$ acting on the common Hilbert space $\mathcal{H}$. If the action $T$ of the translation group $\mathbb{Z}^n$ on $\Omega$ is measure preserving and ergodic, and such that $H_{T_x \omega}$ is unitarily equivalent to $H_\omega$ for every $x \in \mathbb{Z}^n$, then it can be shown that the spectrum $\Spec(H_\omega)$ is almost surely constant \cite{aizenman_warzel}.

For ergodic random operators, the almost sure spectrum can be computed using our method by enumerating all local patches that occur in $H_\omega$ with  non-zero proabability. The ergodicity implies that all such patches occur with probability one. We can thus compute the almost sure spectrum of random operators in a completely deterministic manner, without sampling random potentials, by enumerating all possible local patches. This procedure would have to be optimized for practical use however, because the number of local patches grows exponentially in  $L$ for most random models.

\section{Reduction to finite range operators}
\label{sec-finite-range}

In the problem statements in Definitions~\ref{definitionSpectralProblem} and~\ref{definitionPseudospectralProblem}, we have assumed $H$ to be of short range, meaning that its matrix elements $H_{xy}$ decay with a sufficient power of the distance $d(x,y)$. The statements needed for the computation of the spectrum and pseudospectrum, however, are much easier to prove if we assume $H$ to have finite range. In this section, we show that if we can compute the spectrum for operators of finite range, we can extend this to short-range operators by simply cutting of the operator $H$ at a sufficiently large hopping length.

Let $H$ be a discrete operator on $\mathcal H$, and let $m > 0$. Then we define the \textit{trimmed operator} $G^m$ by its matrix elements
\begin{align}
 G_{xy}^m = \begin{cases}
                     H_{xy} & \text{for }d(x,y) \leq m\,, \\
                     0 & \text{otherwise}\,.
                    \end{cases}
\label{definitionOfTrimming}
\end{align}
The trimmed operator $G^m$, which is also an operator on $\mathcal H$, always has finite range $m$.

As a first step, we show that $G^m \to H$ converges in norm for short-range operators $H$. We can prove this using the Schur test, which provides a bound on the operator norm of a linear operator based on its matrix elements \cite{schur_1911}. A convenient formulation of the Schur test for discrete operators is the following Proposition \cite[Proposition 10.6]{aizenman_warzel}:

\begin{proposition}
\label{propSchurTest}
Let $B$ be a discrete operator on $\Gamma$ satisfying
\begin{align*}
 \norm{B}_{1,1} &:= \sup_{y \in \Gamma} \,\sum_{x \in \Gamma} \, |B_{xy}| < \infty\,,\\
 \norm{B}_{\infty,\infty} &:= \sup_{x \in \Gamma} \,\sum_{y \in \Gamma} \, |B_{xy}| < \infty\,.
\end{align*}
Then $B$ is bounded by $\norm{B} \leq \sqrt{\norm{B}_{1,1}\norm{B}_{\infty,\infty}}$.
\end{proposition}

We will also need the following bound of the sum of $d(x,y)^{-(n+\varepsilon)}$, when summing only over points with a given minimum distance $m$ from $x$.

\begin{restatable}{lemma}{decayLemma}
\label{lemmaDecaySum}
 Let $\varepsilon>0$ and $\Gamma$ be a subset of $\mathbb{R}^n$ which is uniformly discrete with distance $l$. Then there is a constant $C$, depending only on $n$ and $l$, such that
 \begin{align}
  \sum_{\substack{y \in \Gamma \\ d(x,y) > m}} d(x,y)^{-(n+\varepsilon)} \;\leq\; C \cdot \frac1{m^\varepsilon} \,.
  \label{eqDistpowers}
 \end{align}
Furthermore, the constant $C$ can be computed by a BSS algorithm from $n, \varepsilon$ and $l$.
\end{restatable}

The proof of this lemma is given in Appendix~\ref{sec-elementary}. We can now combine Lemma~\ref{lemmaDecaySum} with Proposition~\ref{propSchurTest} to prove the norm convergence $G^m \to H$.

\begin{lemma}
\label{lemmaCutoffDecay}
Let $H$ be a  short-range discrete operator. Then for every $\delta > 0$, there exists an $m > 0$ such that
\begin{align}
\label{eqGmDelta}
 \norm{H - G^m} \leq \delta\,.
\end{align}
The number $m$ is a computable function of $\delta,$ the dimension $n$, and the decay constant $C$ of Definition \ref{definition-short-range}.
\end{lemma}
\begin{proof}
Let $B^m = H - G^m$. The entries of $B^m$ fulfill the opposite relation to \eqref{definitionOfTrimming}:
\begin{align*}
 B^m_{xy} = \begin{cases}
                     H_{xy} & \text{for }d(x,y) > m\,, \\
                     0 & \text{otherwise}\,.
                    \end{cases}
\end{align*}
By the short-range property of $H$, there is a constant $C > 0$ such that
\begin{align*}
 |H_{xy}| \leq C d(x,y)^{-(n+\varepsilon)}
\end{align*}
for all $x, y \in \Gamma$. For any $x \in \Gamma$, we can therefore bound
\begin{align*}
 \sum_{y \in \Gamma} |B_{xy}^m| = \sum_{\substack{y \in \Gamma \\ d(x,y) > m}} |H_{xy}| \leq C \sum_{\substack{y \in \Gamma \\ d(x,y) > m}} d(x,y)^{-(n+\varepsilon)}\,.
\end{align*}
By Lemma~\ref{lemmaDecaySum}, there is a constant $C_2 > 0$, which obviously does not depend on $H$, such that this sum is bounded by
\begin{align*}
 \sum_{\substack{y \in \Gamma \\ d(x,y) > m}} |H_{xy}| \leq C \cdot C_2 \cdot \frac{1}{m^\varepsilon}\,.
\end{align*}
The same bound clearly applies when summing over $x$. In the language of Proposition~\ref{propSchurTest}, this means that
\begin{align*}
 \norm{B}_{1,1} \leq C \cdot C_2 \cdot \frac{1}{m^2} \quad\text{and}\quad \norm{B}_{\infty,\infty} \leq C \cdot C_2 \cdot \frac{1}{m^\varepsilon}\,.
\end{align*}
Therefore Proposition~\ref{propSchurTest} bounds the operator norm of $B^m$ by
\begin{align*}
 \norm{H - G^m} = \norm{B^m} \leq C \cdot C_2 \cdot \frac{1}{m^\varepsilon}\,.
\end{align*}
It follows that Equation \eqref{eqGmDelta} is fulfilled for any $m \geq \left(\frac{C \cdot C_2}{\delta}\right)^{1/\varepsilon}$, proving the Lemma. It is also clear that a suitable $m$ can be computed by a BSS algorithm, since it is given by a simple formula in $C, C_2$ and $\delta$.
\end{proof}

Having shown the convergence $G^m \to H$ in norm, we can now proceed to show how the spectrum of any short-range operator $H$ may be computed from that of a suitable trimming $G^m$. This is quite simple in the normal case, but somewhat more involved for non-normal operators.

\subsection{Normal case}

Lemma~\ref{lemmaCutoffDecay} is sufficient to reduce the computational problem for short-range operators to that for finite range operators. Define the \textit{finite range flc spectral problem} as tuples $(\Omega \,\cap\, \Omega_{\rm fr}, \Lambda \cup \Lambda_{\rm fr}, (\mathcal M, d_H), \Xi)$ and the \textit{finite range flc pseudospectral problem} as $(\Omega_{\rm ps} \cap \Omega_{\rm fr}, \Lambda_{\rm ps} \cup \Lambda_{\rm fr}, (\mathcal M, d_H), \Xi_{\rm ps})$, where $\Omega_{\rm fr}$ is the set of all operators with finite range and $\Lambda_{\rm fr}$ contains just one additional evaluation function that associates to any finite-range operator $H$ the minimal $m$ such that $H_{xy} = 0$ whenever $d(x,y) > m$. Then we have the following proposition:

\begin{proposition}
\label{proposition:FReductionS}
 If the finite range \emph{flc} spectral problem $(\Omega \,\cap\, \Omega_{\rm fr}, \Lambda \cup \Lambda_{\rm fr}, (\mathcal M, d_H), \Xi)$ is solvable, then so is the short-range \emph{flc} spectral problem $(\Omega, \Lambda, (\mathcal{M}, d_H), \Xi)$.
\end{proposition}
\begin{proof}
Suppose that the finite range \emph{flc} spectral problem is solvable. To solve the short-range \emph{flc} spectral problem, let $H \in \Omega$ be given, and let $k \in \mathbb{N}$ be arbitrary. Then by Lemma~\ref{lemmaCutoffDecay}, we can find a cutoff distance $m$ such that $\norm{H - G^m} \leq 2^{-(k+1)}$. Because we have a solution of the finite range \emph{flc} spectral problem, we can then apply that solution to the finite range operator $G^m$ to produce an approximation $A$ to the spectrum of $G^m$ such that
\begin{align*}
d_H(A, \Spec(H)) < 2^{-(k+1)}\,.
\end{align*}
But because $\norm{H - G^m} \leq 2^{-(k+1)}$, we also have
\begin{align*}
d_H(\Spec(H), \Spec(G^m)) \leq 2^{-(k+1)}\,.
\end{align*}
The triangle inequality then implies that $d_H(A, \Spec(H)) \leq 2^{-k}$. Thus, we can always compute a set $A$ such that $d_H(A, \Spec(H)) \leq 2^{-k}$, solving the short-range \textit{flc} spectral problem.
\end{proof}

\subsection{Non-normal case}

For normal operators, we have used that a perturbation of size $\delta$ only changes the spectrum by at most $\delta$ in Hausdorff distance in order to prove that the spectral problem can be reduced to the case of finite range.
For non-normal operators, however, a small perturbation can cause arbitrarily large changes in the $\varepsilon$-pseudospectrum. However, we still have for any   perturbation $R$ of norm $\norm{R} \leq \tau<\varepsilon$ the inclusions
\begin{equation}
\label{eq:Pseudo1}
\Spec_{\varepsilon - \tau}(H) \subseteq \Spec_{\varepsilon}(H + R) \subseteq \Spec_{\varepsilon + \tau}(H)\,.
\end{equation}
  Equation \eqref{eq:Pseudo1} follows immediately from the following equivalent characterization of the $\varepsilon$-pseudospectrum as union of perturbed spectra \cite[Equation 4.4]{trefethenEmbreeBook}:
\begin{align*}
\Spec_\varepsilon(H) = \bigcup_{\substack{B \in \mathcal{L(H)} \\ \norm{B} \leq \varepsilon }} \Spec(H + B)\,.
\end{align*}
To show that the inclusions \eqref{eq:Pseudo1} are sufficient to bound the pseudospectrum of $H + R$ arbitrarily well, we need the following lemma, which shows the continuity of the $\varepsilon$-pseudospectrum in the level $\varepsilon$.

\begin{lemma}
\label{lemmaPreTau}
Let $H$ be a bounded operator and $\varepsilon > 0$. Then for every $\delta > 0$, there exists a $\tau > 0$ so that
\begin{align}
d_\mathrm{H}(\Spec_{\varepsilon - \tau}(H), \Spec_{\varepsilon + \tau}(H)) \leq \delta\,.
\label{eqHausdorffTauEpsilon}
\end{align}
In other words, the function $\varepsilon \to \Spec_{\varepsilon}(H)$ is continuous in $\varepsilon$ if the codomain is equipped with the Hausdorff distance.
\end{lemma}
\begin{proof}
 Let $\delta > 0$, and define the open $\varepsilon$-pseudospectrum
 \begin{align*}
 \spec_\varepsilon(H) := \{ z \in \mathbb{C} \;|\; \rho_H(z) < \varepsilon \}\,.
 \end{align*}
First note that for all $0<\tau<\varepsilon$
\begin{equation}
\begin{aligned}
 &d_\mathrm{H}(\Spec_{\varepsilon - \tau}(H), \Spec_{\varepsilon + \tau}(H)) \leq \delta\\
 &\quad\qquad\Leftrightarrow\quad D_\tau := \Spec_{\varepsilon + \tau}(H) \cap \left( B_\delta(\Spec_{\varepsilon - \tau}(H)) \right)^c = \varnothing\,.
 \end{aligned}
 \label{ineq_reformulated}
\end{equation}
In this proof we use the notation $B_\delta(A):= \bigcup_{z\in A} B_\delta(z)$ for $A\subset\mathbb{C}$ and $B_\delta(z)$ just denotes the open ball with radius $\delta$ around $z$. Since $\Spec_{\varepsilon}(H)$ is closed and bounded,
the sets $D_\tau$ are compact for all $\tau > 0$. Their intersection is
\begin{align*}
 \bigcap_{\tau > 0} D_{\tau} &
= \;\bigcap_{\tau > 0} \left[ \Spec_{\varepsilon + \tau}(H) \cap \left( B_\delta(\Spec_{\varepsilon - \tau}(H)) \right)^c \right] \\
 &=\; \left[ \bigcap_{\tau > 0} \Spec_{\varepsilon + \tau}(H) \right] \cap \left[ \bigcap_{\tau > 0} \left( B_\delta(\Spec_{\varepsilon - \tau}(H)) \right)^c \right] \\
 &=\; \Spec_\varepsilon(H) \cap \left[ \bigcup_{\tau > 0} B_\delta(\Spec_{\varepsilon - \tau}(H)) \right]^c \\
 &=\; \Spec_\varepsilon(H) \cap B_\delta \left( \bigcup_{\tau > 0} \Spec_{\varepsilon - \tau}(H) \right)^c \\
 &=\; \Spec_\varepsilon(H) \cap B_\delta \left( \spec_\varepsilon(H) \right)^c
\end{align*}
Now
\begin{align*}
B_\delta(\spec_\varepsilon(H)) \supseteq \overline{\spec_\varepsilon(H)} = \Spec_\varepsilon(H)\,,
\end{align*}
where the last equality follows from the fact that   for linear operators $H$ on Hilbert spaces the function  $\rho(H)\ni z\mapsto \|(H-z)^{-1}\|$ and thus also the function $\rho(H)\ni z\mapsto \rho_H(z)$ cannot be constant on any open set, see Proposition~2 in~\cite{globevnik_1976}.
Hence, $\bigcap_{\tau > 0} D_{\tau} = \varnothing$. Since the sets $D_{\tau}$ are compact and decreasing (that is, $D_{\tau'} \subseteq D_\tau$ for $\tau' \geq \tau$), there exists $\tau > 0$ such that $D_\tau = \varnothing$. This proves \eqref{ineq_reformulated} and thereby the lemma.
\end{proof}

Lemma \ref{lemmaPreTau} will also be useful to prove the computability of the pseudospectrum for finite range operators in Section~\ref{sec-computability-pseudospectrum}. Here we use it  to prove that the computation of the pseudospectrum can be reduced to the case of finite range operators instead of short-range ones using the inclusions \eqref{eq:Pseudo1}.

\begin{proposition}
\label{prop:FRreductionPS}
If the finite range \emph{flc} pseudospectral problem $(\Omega_{\rm ps} \cap \Omega_{\rm fr}, \Lambda_{\rm ps} \cup \Lambda_{\rm fr}, (\mathcal M, d), \Xi_{\rm ps})$ is solvable, then the short-range problem $(\Omega_{\rm ps}, \Lambda_{\rm ps}, (\mathcal M, d), \Xi_{\rm ps})$ is solvable.
\end{proposition}
\begin{proof}
Assume that the finite range problem is solvable. Let $(H, \varepsilon) \in \Omega_{\rm ps} \cap \Omega_{\rm fr}$ be a concrete short-range spectral problem and fix the allowed margin of error $\delta:=2^{-k}$. In order to approximate the $\varepsilon$-pseudospectrum of $H$, we proceed as follows: Iterate over $\tau = 1, \frac12, \frac14, \dots$. For each $\tau$, trim $H$ to an operator $G^m$ such that $\norm{H-G^m} \leq \tau$. A cut-off length $m$ fulfilling that inequality can be found in a computable way by Lemma~\ref{lemmaCutoffDecay}.
We have assumed that the finite range problem has a solution, so we can now apply that solution to construct subsets $A$ and $B$ such that
\begin{gather}
d_\mathrm{H}(A, \Spec_{\varepsilon - \tau}(G^m)) \leq \delta / 6 \label{eq:approximationPropertyA}\phantom{\,.}\\
 d_\mathrm{H}(B, \Spec_{\varepsilon + \tau}(G^m)) \leq \delta / 6\label{eq:approximationPropertyB}\,.
\end{gather}
Then, if we have $d_\mathrm{H}(A, B) \leq \delta / 2$, we terminate with $A$ as our approximation to $\Spec_\varepsilon(H)$, otherwise we continue with the next $\tau$ in our sequence.

Now if the algorithm terminates because the condition $d_\mathrm{H}(A, B) \leq \delta / 2$ is fulfilled, then the result $A$ will be an approximation of $\Spec_\varepsilon(H)$ with Hausdorff distance less than $\delta$. We can show this using the triangle inequality and the inclusions \eqref{eq:Pseudo1} as follows:
\begin{align*}
d_\mathrm{H}(A, \Spec_\varepsilon(H)) &\leq
d_\mathrm{H}(A, \Spec_{\varepsilon - \tau}(G^m)) + d_\mathrm{H}(\Spec_{\varepsilon - \tau}(G^m), \Spec_\varepsilon(H))
\\
&\leq
d_\mathrm{H}(A, \Spec_{\varepsilon - \tau}(G^m)) + d_\mathrm{H}(\Spec_{\varepsilon - \tau}(G^m), \Spec_{\varepsilon + \tau}(G^m)) \\
&\leq d_\mathrm{H}(A, \Spec_{\varepsilon - \tau}(G^m)) + d_\mathrm{H}(\Spec_{\varepsilon - \tau}(G^m), A)\\
&\phantom{\leq } + d_\mathrm{H}(A, B) + d_\mathrm{H}(B, \Spec_{\varepsilon + \tau}(G^m)) \\
&\leq 3 \tfrac{\delta}{6} + \tfrac{\delta}{2} = \delta \,.
\end{align*}
In the last step, we have combined the cut-off bounds \eqref{eq:approximationPropertyA} and \eqref{eq:approximationPropertyB} with the condition $d_H(A, B)$.

It remains to show that the described algorithm does always terminate. That is, we need to show that the condition
$d_\mathrm{H}(A, B) \leq \delta / 2$
will be fulfilled  for $\tau$ small enough and $A, B$ as above. According to Lemma~\ref{lemmaPreTau} there exists a $\tau' > 0$ such that
\begin{align*}
d_\mathrm{H}(\Spec_{\varepsilon - \tau'}(H), \Spec_{\varepsilon + \tau'}(H)) \leq \delta / 6\,.
\end{align*}
 This does not imply a way of computing $\tau'$, but because our algorithm descends to arbitrarily small $\tau$, it will eventually come to a $\tau$ such that $\tau < \tau' / 2$.
Assuming this, the algorithm will then choose $G^m$ such that $\norm{G^m - H} < \tau$, and we can use  \eqref{eq:Pseudo1} to conclude
\begin{align}
\Spec_{\varepsilon - \tau'}(H) \subseteq \Spec_{\varepsilon - \tau}(G^m) \subseteq \Spec_{\varepsilon}(H) \subseteq \Spec_{\varepsilon + \tau}(G^m) \subseteq \Spec_{\varepsilon + \tau'}(H)\,,
\label{eqFivespec}
\end{align}
where we have used $2\tau < \tau'$.
We can now apply Lemma~\ref{lemmaABCX} twice to the inclusions in~\eqref{eqFivespec} to see that
\begin{align*}
d_\mathrm{H}( \Spec_{\varepsilon + \tau}(G^m),\Spec_{\varepsilon - \tau}(G^m)) &\leq d_\mathrm{H}( \Spec_{\varepsilon + \tau}(G^m),\Spec_{\varepsilon - \tau'}(H)) \\
&\leq d_\mathrm{H}(\Spec_{\varepsilon - \tau'}(H), \Spec_{\varepsilon + \tau'}(H)) \\
&\leq \delta / 6\,.
\end{align*}
It then follows from  \eqref{eq:approximationPropertyA} and~\eqref{eq:approximationPropertyB} that
\begin{align*}
d_\mathrm{H}(A, B) &\leq d_\mathrm{H}(A, \Spec_{\varepsilon - \tau}(G^m)) + d_\mathrm{H}(\Spec_{\varepsilon - \tau}(G^m), \Spec_{\varepsilon + \tau}(G^m))+ d_\mathrm{H}(\Spec_{\varepsilon + \tau}(G^m), B)\\
&\leq 3  \frac{\delta}{6} = \delta / 2\,.
\end{align*}
We have thus shown that the algorithm terminates at the latest when $\tau$ becomes smaller than $\tau'/2$.
\end{proof}

We have thus established the reduction of the spectral and pseudospectral problems to finite range operators. In the rest of the paper, we  restrict ourselves to finite range operators when describing how to solve the \emph{flc} spectral and pseudospectral problems, since by Propositions~\ref{proposition:FReductionS} and~\ref{prop:FRreductionPS}, the same algorithms can then also be applied to the short-range \emph{flc} problems by using a suitable cut-off.

\section{The pseudospectral inclusion bound}
\label{sec-upper-bound}

In this section, we formulate and prove the pseudospectral inclusion bound from \cite{colbrook_2019} for our setting. Moreover, we also show that the bound is computable.

\subsection{Formulation of the pseudospectral inclusion bound}
\label{subsec-inclusion-bound}

A bound showing the inclusion of a complex number in  the $\varepsilon$-pseudospectrum for a certain~$\varepsilon$, or in other words an upper bound on $\rho_H$, has been described in \cite{colbrook_2019}.
In the following  we describe the bound using the singular value decomposition (SVD) of a rectangular matrix $A$, while in \cite{colbrook_2019}, the eigenvalues of $A^T A$ are used, which are just the squares of the singular values of $A$.

\begin{definition}\label{def:Q}
Let $H$ be a discrete operator with finite range $m$, let $x \in \mathbb{R}^n$ be any point and $L > 0$, $\lambda \in \mathbb{C}$. The   \textit{uneven section} for these data is the linear operator
\begin{gather*}
  Q_{L, \lambda,x} : \mathcal{H}_{B_L(x)} \to \mathcal{H}_{B_{L+m}(x)} \\
  Q_{L, \lambda,x}(\psi) = \mathbf{1}_{B_{L+m}(x)} (H - \lambda) \mathbf{1}_{B_L(x)} \psi\,.
\end{gather*}
\end{definition}

Recall that   $\mathcal{H}_{B_L(x)} := \mathrm{span}\{ e_y\,|\, y\in B_L(x)\cap \Gamma\}$ and $\mathbf{1}_{B_L(x)}$ denotes the orthogonal projection onto $\mathcal{H}_{B_L(x)}$. Because $\Gamma$ is uniformly discrete, $\mathcal{H}_{B_L(x)}$ and $\mathcal{H}_{B_{L+m}(x)}$ are finite dimensional, so that $Q_{L,\lambda,x}$ can be respresented as a rectangular matrix.

\begin{theorem}
\label{upperBoundTheorem}
 Let $H$ be a discrete operator with finite range $m >0$, and let $\lambda \in \mathbb{R}$, $L > 0$, and $x \in \mathbb{R}^n$ be arbitrary. Let
 \begin{align}
 \label{eqDefinitionEpsilon}
 \varepsilon_{L,\lambda,x} = s_1(Q_{L, \lambda, x})
 \end{align}
 be the smallest singular value of $Q_{L,\lambda,x}$. Then $\lambda$ is contained in the $\varepsilon_{L, \lambda, x}$-pseudospectrum of~$H$:
 \begin{align}
 \label{eqRhoUpperBound}
 \rho_H(\lambda)\leq \varepsilon_{L, \lambda, x}\,.
 \end{align}
\end{theorem}

\begin{proof}
The singular value decomposition (SVD) of $Q_{L,\lambda,x}$ decomposes $Q_{L, \lambda, x}$ as a product
\begin{align*}
Q_{L, \lambda,x} = U S V^*
\end{align*}
where $U$ and $V$ are unitary and $S$ is a rectangular diagonal matrix. We can also write this matrix product as
\begin{align*}
 Q_{L, \lambda, x} = \sum_{i = 1}^n |u_i \rangle s_i \langle v_i |
\end{align*}
where $(u_i)_{i = 1, \dots, n}$ are the first $n$ columns of $U$, $(v_i)_{i = 1, \dots, n}$ are the columns of $V$, and $(s_i)_{i = 1, \dots, n}$ are the diagonal elements of $S$. We assume that the $(s_i)_{i = 1, \dots, n}$ are sorted in ascending order so that $s_1 = \varepsilon_{L, \lambda, x}$.

Then $v_1$, considered as a vector in $\mathcal H$ via the canonical inclusion $\iota : \mathcal H_{B_L(x)} \to \mathcal H$, is an $\varepsilon_{L,\lambda,x}$-quasimode for $H$. Indeed, we have
\begin{align*}
 (H - \lambda) v_1 &= (H - \lambda) \mathbf{1}_{B_L(x)} v_1 &\text{(since $v_1$ is supported on $B_L(x)$)} \\
 &= \mathbf{1}_{B_{L+m}(x)} (H - \lambda) \mathbf{1}_{B_L(x)} v_1 &\text{(by the finite range of $H$)} \\
 &= Q_{L, \lambda, x} v_1 & \\
 &= s_1 u_1 &
\end{align*}
and therefore $\norm{(H - \lambda) v_1} = s_1 \norm{u_1} = s_1$. We conclude that $\norm{(H - \lambda)v_1} \leq s_1 =  \varepsilon_{L,\lambda, x}$, and thus  $v_1$ is an $\varepsilon_{L,\lambda,x}$-quasimode, which implies $\lambda \in \Spec_{\varepsilon_{L, \lambda, x}}(H)$ by definition of the pseudospectrum~\cite{trefethenEmbreeBook}.
\end{proof}

For normal operators, Equation \eqref{eqRhoUpperBound} corresponds to an upper bound on the distance to the spectrum of $H$.

\subsection{Computability of the pseudospectral inclusion bound}

In order to derive statements about computability from the existence of the pseudospectral inclusion bound, we need to relate it to our model of computation and show that the bound is actually computable from our problem definition. The following Lemma shows the computability of $\varepsilon_{L, \lambda, x}$ for a fixed $x \in \mathbb{R}^n$. We recall that a short review on BSS algorithms is given in Appendix \ref{appendixBSS}.

\begin{lemma}
\label{lemmaComputabilityUpperBound}
 For every $k\in\mathbb{N}$ and $\varepsilon>0$  there exists a   BSS algorithm $\Gamma_{\varepsilon, k}$ such that $\Gamma_{\varepsilon, k}(Q_{L, \lambda, x}) \in \mathbb{R}$ and
\begin{align}
d(\Gamma_{\varepsilon, k}(Q_{L, \lambda, x}), \varepsilon_{L, \lambda, x}) \leq 2^{-k}\,.
\label{eqGammaEpsilonClose}
\end{align}
\end{lemma}
\begin{proof}
First, let us note that in \cite{colbrook_2019} the authors describe a method for approximating $\varepsilon_{L, \lambda, x}$ using finitely many arithmetic operations. In short, to compute a bound on the distance to the spectrum using Proposition~\ref{upperBoundTheorem}, one can use that $\varepsilon_{L, \lambda, x}$ is the square root of the smallest eigenvalue of $Q_{L, \lambda, x}^T Q_{L, \lambda, x}$. This implies that, for any $t \in \mathbb{R}$,
\begin{align*}
 t < \varepsilon_{L, \lambda, x}
 \quad\Leftrightarrow\quad
 Q_{L, \lambda, x}^T Q_{L, \lambda, x} - t^2 \text{ is positive definite.}
\end{align*}

Whether a symmetric matrix is positive definite can be decided using finitely many rational computations using the Cholesky decomposition. Recall that the Cholesky decomposition is a method of factorizig positive semidefinite matrices $A$ in the form $L^*L$, where $L$ is a lower triangular matrix. If the matrix is not positive semidefinite, this will lead to the square root of a negativ number being computed in the algorithm. Thus, we can approximate $\varepsilon_{L, \lambda, x}$ with finitely many operations.

In view of this, the formulation of the BSS method $\Gamma_{\varepsilon, k}$ is simple because for every $t$, testing whether $t < \varepsilon_{L, \lambda, x}$ is computable using a finite number of additions, multiplications and divisions. To define $\Gamma_{\varepsilon, k}$, let $S = \{ 2^{-j}\,|\, j\in \mathbb{N} \}$. The algorithm $\Gamma_{\varepsilon, k}$ can then simply check for each $t \in S$ successively whether $Q_{L, \lambda, x}^T Q_{L, \lambda, x} - t^2$ is positive definite or not, terminating with the smallest $t \in S$ for which this is not the case. (Since $\varepsilon_{L, \lambda, x}$ is a real number, there must certainly be a $t \in S$ with $t \geq \varepsilon_{L, \lambda, x}$.) We take this $t$ as the result of the BSS algorithm $\Gamma_{\varepsilon, k}$. Since $t - 1/k < \varepsilon_{L, \lambda, x} \leq t$, we know that $|t - \varepsilon_{L, \lambda, k}| < \varepsilon$, proving that \eqref{eqGammaEpsilonClose} is satisfied.
\end{proof}

\section{Spectral gap bound}
\label{sec-quasimode-detectability}

It was proven in \cite{colbrook_2019} that the pseudospectral inclusion bound $\varepsilon_{L, \lambda, x}$ always converges to $\rho_H(\lambda)$ as $L \to \infty$, at least for normal, short-range operators. However, the convergence is not uniform in the operator $H$, and thus cannot be used to provide a lower bound on $\rho_H(\lambda)$. As described in the introduction, we can work around this by instead using $\varepsilon_{L, \lambda}$, which is defined as the infimum of $\varepsilon_{L, \lambda, x}$ over all $x \in \Rn$. In this section, we prove Theorem~\ref{introThmNew} by constructing, for any quasimode $\psi$ of the infinite-volume operator $H$, a quasimode of $H$ that is restricted to a square $B_L(x)$. Taking inspiration from~\cite{chonchaiya_2011,chandlerwilde_2023,chandlerwilde_2024,lindner_unpublished}, we do so by multiplying $\psi$ with a compactly supported function of the following form.

\begin{definition}
For every $x \in \Rn$ and $L>0$ we define the \textit{tent function} $V_{L, x} : \Rn \to \mathbb{R}$ by
\begin{align*}
 V_{L, x}(y) := \max\left(0,1- \frac{\norm{x-y}}{L}\right).
\end{align*}
\end{definition}

Recall that we are using the maximum norm on $\Rn$, giving the tent a square base. For all $L > 0$ and $x \in \Rn$ the function $V_{L,x}$ is  Lipschitz continuous with Lipschitz constant $1/L$.

\begin{remark}
The simple tent function $V_{L,x}(y)$ is sufficient to prove the bound we are interested in in the form needed for our computability proof. For a  more complicated choice that leads to a  smaller constant $C$ in Theorem~\ref{introThmNew}, see \cite{chandlerwilde_2024}.
\end{remark}

In the following lemma, we provide a quantitative bound on the commutator of the tent function $V_{L,x}$ with a finite range operator $H$ averaged over $x \in \mathbb R^n$.

\begin{lemma}
\label{lemma-bound-commutator}
Let $\Gamma$ and  $H$ be as in Theorem~\ref{introThmNew}.
Then for  all $\psi \in \mathcal H$
\begin{align}
\label{eq-commutator-bound}
 \int_{x \in \Rn} \norm{[V_{L,x}, H] \psi}_{\mathcal H}^2\, \mathrm{d}x^n \leq \frac{ m^2\,M^2}{L^2}  \left( \frac{36m}{q} \right)^n\norm{V_{L, 0}}^2_{L^2(\Rn)}\,\norm{\psi}_{\mathcal H}^2 \,.
\end{align}
\end{lemma}

\begin{proof}
Let $\psi \in \mathcal H$.
Then
\begin{align*}
R \,:=&  \,\int_{x \in \mathcal \Rn} \norm{[V_{L, x}, H] \psi}_{\mathcal H}^2 \, \mathrm{d}x^n= \int_{x \in \mathcal \Rn} \sum_{p \in \Gamma} |([V_{L, x}, H] \psi)(p)|^2\, \mathrm{d}x^n \\ =& \, \int_{x \in \mathcal \Rn} \sum_{p \in \Gamma} \left| \sum_{q \in \Gamma} (V_{L, x}(p) - V_{L, x}(q)) H_{pq} \psi(q)\right|^2  \mathrm{d}x^n\\
=& \, \int_{x \in \mathcal \Rn} \sum_{p \in \Gamma} \sum_{q_1 \in \Gamma} \sum_{q_2 \in \Gamma} |V_{L, x}(p) - V_{L, x}(q_1)| |H_{p q_1}| |\psi(q_1)| |V_{L, x}(p) - V_{L, x}(q_2)| |H_{p q_2}| |\psi(q_2)|\,\mathrm{d}x^n\,.
\end{align*}
Using that $H$ has finite range $m$ we conclude that  in the above sum only terms with $d(p, q_1) < m$ and $d(p, q_2) < m$ can contribute.
Moreover, at least one of the points $p$ and $q_1$ and at least one of $p$ and $q_2$ must be in $B_L(x)$, because otherwise, $|V_{L,x}(p) - V_{L,x}(q_1)| = 0$ or $|V_{L,x}(p) - V_{L,x}(q_2)| = 0$, respectively. Hence, only terms with $p \in B_{L + m}(x)$ can contribute.
Let
\begin{align*}
S_{L, x} := \left\{(p, q_1, q_2) \in \Gamma^3 \;\middle|\; d(p, q_1) < m \;\text{and}\; d(p, q_2) < m\;\text{and}\;d(p, x) < L + m\right\}\,,
\end{align*}
then
\begin{align*}
R &= \int_{x \in \mathcal \Rn} \sum_{(p,q_1,q_2) \in S_{L,x}} |V_{L, x}(p) - V_{L, x}(q_1)| |H_{p q_1}| |V_{L, x}(p) - V_{L, x}(q_2)| |H_{p q_2}| |\psi(q_1)| |\psi(q_2)|\,\mathrm{d}x^n.
\end{align*}
Using Lipschitz continuity of  $V_{L, x}$ and $|H_{yz}|\leq M$, we obtain
\begin{align*}
R \,&\leq\, \frac{m^2\,M^2}{L^2} \int_{x \in \mathcal \Rn} \sum_{(p,q_1,q_2) \in S_{L,x}}   |\psi(q_1)| |\psi(q_2)|\,\mathrm{d}x^n.
\end{align*}
At this point, the summands do not depend on the index $p$ any more. Thus, let
\begin{align*}
T_{L, x} := \{ (q_1, q_2) \in \Gamma^2 \;|\; \exists p \in \Gamma \;\text{with}\; (p, q_1, q_2) \in S_{L,x} \}
\end{align*}
be the projection onto the latter two indices. Now for each pair $(q_1, q_2) \in T_{L, x}$, all elements $p \in \Gamma$ for which $(p, q_1, q_2) \in S_{L, x}$ must lie in the intersection
\begin{align*}\
 p \in \Gamma \cap B_m(q_1) \cap B_m(q_2) \subseteq \Gamma \cap B_{2m}(\bar q)\,,
\end{align*}
where $\bar q = \frac12(q_1 + q_2)$, by the triangle inequality. Thus, the number of such $p \in \Gamma$ is bounded by
\begin{align*}
 \# \{ p \in \Gamma \;|\; (p, q_1, q_2) \in S_{L, x} \} \leq \frac{\vol B_{2m}(0)}{\vol B_{r/2}(0)} = \left( \frac{4m}{q} \right)^n =: \gamma\,,
\end{align*}
where $q$ is the infimal distance between points of $\Gamma$. We thus have
\begin{align*}
 R \,\leq\, \frac{\gamma\,m^2\,M^2}{L^2}  \int_{x \in \mathcal \Rn} \sum_{(q_1,q_2) \in T_{L,x}} |\psi(q_1)| |\psi(q_2)|\,\mathrm{d}x^n.
\end{align*}
Now for each pair of points $(q_1, q_2) \in T_{L,x}$, we must have $d(x, q_1) < L + 2m$ and $d(x, q_2) < L + 2m$, by the triangle inequality. Thus, we can estimate the sum from above by
\begin{align*}
 R \,\leq&\; \frac{\gamma\,m^2\,M^2}{L^2} \int_{x \in \mathcal \Rn} \sum_{q_1,q_2\in B_{L+2m}(x)}
 |\psi(q_1)| |\psi(q_2)|\,\mathrm{d}x^n \\
 =&\; \frac{\gamma\,m^2\,M^2}{L^2}  \int_{x \in \mathcal \Rn} \left(\sum_{q \in B_{L + 2m}(x)} |\psi(q)| \right)^2\mathrm{d}x^n  \\
 =&\; \frac{\gamma\,m^2\,M^2}{L^2} \, \norm{|\psi| * \mathbf{1}_{B_{L+2m}(0)}}_{L^2(\Rn)}^2\,,
\end{align*}
where we define the convolution as an operator $* : \ell^2(\Gamma) \times L^2(\Rn) \to L^2(\Rn)$. We can now apply Young's inequality to obtain
\begin{align*}
R \,\leq\, \frac{\gamma\,m^2\,M^2}{L^2}\; \norm{|\psi|}_{\mathcal H}^2\, \norm{\mathbf{1}_{B_{L+2m}(0)}}_{L^1(\Rn)}^2\,=\,
\frac{\gamma\,m^2\,M^2}{L^2}\; \norm{\psi}_{\mathcal H}^2\,(2L + 4m)^n
\,.
\end{align*}
On the other hand, we have
\begin{align*}
\norm{V_{L, 0}}_{L^2(\Rn)}^2 = \left(\norm{V_{L, 0}}_{L^2(\mathbb R)}^2\right)^n = \left(\tfrac23 L\right)^n\,,
\end{align*}
and thus, using $L > m$, that
\begin{align*}
\frac{\norm{\mathbf{1}_{B_{L+2m}(0)}}_{L^1(\Rn)}^2}{\norm{V_{L, 0}}_{L^2(\Rn)}^2} =
\frac{(2L + 4m)^n}{\left(\frac23 L\right)^n}
= 3^n (1+ 2\tfrac{m}{L})^n < 9^n\,.
\end{align*}
In coclusion, we have
\[
R \,\leq\, \frac{9^n\,\gamma\,m^2\,M^2}{L^2} \norm{\psi}_{\mathcal H}^2 \norm{V_{L, 0}}^2_{L^2(\Rn)}\,.\qedhere
\]
\end{proof}

We now show how the approximate commutation of $V_{L,x}$ and $H$ in the average over $x \in \Rn$ implies that a given quasimode can be modified into a local quasimode by multiplication by one of the $V_{L, x}$.

\begin{lemma}
\label{lemma-use-commutator-bound}
Let $\Gamma$ and $H$ be as in Theorem~\ref{introThmNew}. Let $\lambda \in \mathbb{C}$, $\varepsilon > 0$ and suppose that $\psi \in \mathcal H$ is such that
\begin{align*}
 \norm{(H - \lambda)\psi} \leq \varepsilon \norm{\psi}\,.
\end{align*}
Then there exists   $x \in \Rn$ such that
\begin{align*}
\norm{(H - \lambda) V_{L,x} \psi} \leq \left(\varepsilon + \tfrac{C}{L}\right) \norm{\psi}\,,
\end{align*}
with $C=m\,M   \left( \frac{36m}{q} \right)^{n/2}$,
and thus
\begin{align*}
 \varepsilon_{L, \lambda} \leq \varepsilon_{L, \lambda, x} \leq \varepsilon + \tfrac{C}{L}\,.
\end{align*}
\end{lemma}

\begin{proof}
We first note that for any $\phi \in \mathcal H$, we have
\begin{align}
 \int_x \norm{V_{L, x}\phi}_{\mathcal H}^2\,\mathrm{d}x^n = \norm{\phi}_{\mathcal H}^2\, \norm{V_{L,0}}_{L^2(\Rn)}^2\,.
 \label{eq-int-Vx}
\end{align}
By Minkowski's inequality,
\begin{align*}
 \int_{x \in \Rn} \norm{(H - \lambda )& V_{L,x} \psi}_{\mathcal H}^2  \,\mathrm{d}x^n  \;=\;  \int_{x \in \Rn} \norm{V_{L,x} (H - \lambda) \psi - [V_{L,x}, H] \psi}_{\mathcal H}^2\,\mathrm{d}x^n \\
&\;\leq\; \left(\sqrt{\int_{x \in \Rn} \norm{V_{L,x} (H - \lambda) \psi}^2\,\mathrm{d}x^n } + \sqrt{\int_{x \in \Rn} \norm{[V_{L,x}, H] \psi}^2\,\mathrm{d}x^n }\right)^2.
\end{align*}
Applying \eqref{eq-int-Vx} to the first term, we find
\begin{align*}
\sqrt{\int_{x \in \Rn} \norm{V_{L,x} (H - \lambda ) \psi}_{\mathcal H}^2\,\mathrm{d}x^n } &= \norm{V_{L,0}}_{L^2(\Rn)} \norm{(H - \lambda ) \psi}_{\mathcal H} \\
&\leq \norm{V_{L,0}}_{L^2(\Rn)}\, \varepsilon\, \norm{\psi}_{\mathcal H}\,.
\end{align*}
Combining this with \eqref{eq-commutator-bound} applied to the second term, we obtain
\begin{align*}
\sqrt{\int_{x \in \Rn} \norm{(H - \lambda ) V_{L,x} \psi}^2\,\mathrm{d}x^n } \leq (\varepsilon + \tfrac{C}{L}) \norm{V_{L,0}}_{L^2(\Rn)} \norm{\psi}_{\mathcal H}\,.
\end{align*}
We can now apply \eqref{eq-int-Vx} again to get
\begin{align*}
 \int_{x \in \Rn} \norm{(H - \lambda ) V_{L,x} \psi}^2 \,\mathrm{d}x^n \;\leq\; \int_{x \in \Rn} (\varepsilon + \tfrac{C}{L})^2 \norm{V_{L,x} \psi}^2\,\mathrm{d}x^n \,.
\end{align*}
It follows that there exists an $x \in \mathbb{R}^n$ such that
\begin{align*}
\norm{(H - \lambda ) V_{L,x} \psi}^2 \leq (\varepsilon + \tfrac{C}{L})^2 \norm{V_{L,x} \psi}^2
\end{align*}
and thus
\begin{align*}
\norm{(H - \lambda) V_{L,x} \psi} \leq (\varepsilon + \tfrac{C}{L}) \norm{V_{L,x} \psi}\,.
\end{align*}
The function $ \tilde \psi := V_{L, x} \psi$ lies in $\mathcal{H}_{B_L(x)}$
and    applying to $\tilde \psi$ the operator $Q_{L, \lambda, x}$ of Definition~\ref{def:Q} we obtain
\begin{align*}
 \norm{Q_{L, \lambda, x} \tilde \psi} = \norm{(H - \lambda ) \tilde\psi}_{\mathcal H} \leq (\varepsilon + \tfrac{C}{L}) \norm{\tilde \psi}_{\mathcal H}\,.
\end{align*}
Hence, by the definition of $\varepsilon_{L, \lambda, x}$ and  $\varepsilon_{L, \lambda}$, it follows that
\[
\varepsilon_{L, \lambda}\leq  \varepsilon_{L, \lambda, x} \leq \varepsilon + \tfrac{C}{L}\,.\qedhere
\]
\end{proof}

We can now combine these results to prove Theorem~\ref{introThmNew}.

\begin{proof}\textit{of Theorem~\ref{introThmNew}.}
 Fix $\lambda \in \mathbb C$. For any $\varepsilon > \rho_H(\lambda)$, by the definition of the lower norm function, we can find a quasimode $\psi \in \mathcal H$ such that
\begin{align*}
 \norm{(H - \lambda)\psi} \leq \varepsilon \norm{\psi}\,.
\end{align*}
  Lemma~\ref{lemma-use-commutator-bound} then implies
\begin{align*}
 \varepsilon_{L, \lambda} \leq \varepsilon + \tfrac{C}{L}\,.
\end{align*}
Since this holds for all $\varepsilon > \rho_H(\lambda)$, we obtain
\begin{align*}
 \varepsilon_{L, \lambda} - \tfrac{C}{L} \;\leq\; \rho_H(\lambda)\,,
\end{align*}
  proving the theorem.
\end{proof}

\section{Computability of \texorpdfstring{$\rho_H$}{rho(H)}}
\label{sec-computability-lower-norm}

In this section, we show that the lower norm function $\rho_H(\lambda)$ can be computed with error control, for any \emph{flc} operator $H$ and $\lambda \in \mathbb{C}$. In fact this is a fairly straightforward corollary of Theorems~\ref{introThmNew} and~\ref{upperBoundTheorem}.

\begin{corollary}\label{approximability}
Let $\lambda \in \mathbb{C}$ and $\tau > 0$. Then there exists a BSS algorithm which, given an \emph{flc} discrete operator $H$ of finite range, and using the evaluation functions $\Lambda_{\rm fr}$ of Proposition~\ref{proposition:FReductionS}, computes a value $\tilde\rho_H(\lambda, \tau)$ such that
\begin{align}
\label{eqWhatTildeRhoFulfills}
|\rho_H(\lambda) - \tilde\rho_H(\lambda, \tau)| \leq \tau\,.
\end{align}
\end{corollary}
\begin{proof}
For any given operator $H$, we can use the evaluation function $f_6(H)$ defined in Condition~\ref{definitionEvaluationFunctions}   to compute an upper bound on $|H_{xy}|$ as follows:
\begin{align*}
|H_{xy}| \leq M d(x,y)^{-(n+\varepsilon)} \leq M = f_6(H)\qquad \mbox{ for all }x,y\in\Gamma \,.
\end{align*}
Note that this bound holds uniformly also
after the truncation procedure  described in Section~\ref{sec-finite-range}. Using the enumeration of local patches given by the functions $f_{1,L}, f_{2,L,m}$ and $f_{4,L,n,k,i}(H)$ for a large enough $L$, the packing distance $q$ of the underlying point set $\Gamma$ can also be computed.

We can thus compute everything we need to determine the constant $C$ of Theorem~\ref{introThmNew}. Having determined $C$ for a given operator $H$, we then set
\begin{align*}
L_c := \frac{C}{\tau}
\end{align*}
and compute
\begin{align*}
\tilde \rho_{H}(\lambda, \tau) := \varepsilon_{L_c, \lambda}\,.
\end{align*}
We have established in Lemma~\ref{lemmaComputabilityUpperBound} that $\varepsilon_{L, \lambda}$ can be computed by a BSS algorithm using the evaluation functions of Condition~\ref{definitionEvaluationFunctions}.

It now follows from Theorems~\ref{introThmNew} and~\ref{upperBoundTheorem} that we have
\begin{align*}
\varepsilon_{L_c, \lambda} - \frac{C}{L_c} \leq \rho_H(\lambda) \leq \varepsilon_{L_c, \lambda}.
\end{align*}
Thus,
\begin{align*}
|\rho_H(\lambda) - \tilde\rho_H(\lambda, \tau)| \leq \frac{C}{L_c} = \tau\,,
\end{align*}
proving the corollary.
\end{proof}

Our algorithms for computing the spectrum and $\varepsilon$-pseudospectrum of \textit{flc} operators $H$ will be based on approximating $\rho_H(\lambda)$ at different points $\lambda \in \mathbb{C}$.
In fact, the computability of approximations $\tilde\rho_{H}(\lambda,\tau)$ fulfilling \eqref{eqWhatTildeRhoFulfills} are all that we use about the operator $H$ in showing the computability of the (pseudo-) spectrum. The same methods could therefore also be applied to show the computability of the (pseudo-) spectrum with Hausdorff error control for any other class of operators where the lower norm function $\rho_H(\lambda)$ can be approximated with error control.

\section{Computability of the spectrum for normal operators}
\label{sec-computability-spectrum}

In this section, we will prove Theorem~\ref{theoremComputabilitySpectrum}, the computability of the spectrum for normal operators. Our proof will be based on Corollary~\ref{approximability}, the computability of $\rho_H(\lambda)$. Note that the normal spectral problem is not a special case of the pseudospectral problem because we can only compute the $\varepsilon$-pseudospectrum for $\varepsilon > 0$. In addition to normal operators, the method described below could easily be extended, as in \cite{colbrook_2019}, to classes of operators fulfilling a bound of the form
\begin{align}
  \rho_H(\lambda) \geq g(d(\lambda, \Spec(H)))
\end{align}
for some strictly increasing continuous function $g$.

\begin{proof}[Proof of Theorem~\ref{theoremComputabilitySpectrum}]

Let the operator $H$ and $k \in \mathbb{N}$ be given. We define the desired accuracy as $\tau := 2^{-k}$.

We know that $H$ is bounded, and the algorithm can access a bound on the norm of $H$ via the evaluation function $f_5 \in \Lambda$ (see Definition~\ref{definitionEvaluationFunctions}). Let $M := f_5(H)$ be this bound and define  the restricted  square grid
\begin{align}
\label{eqDefinitionTtau}
  T_\tau := B_M(0) \cap \left(\frac\tau4 \sqrt{2}\cdot \mathbb{Z}^2\right) \subseteq \mathbb{C}\,,
\end{align}
 where we use the natural identification $\mathbb{R}^2 \simeq \mathbb{C}$.
The side length of the rectangular grid is chosen such that $T_\tau$ has covering radius $\tau/4$ inside $B_M(0)$. We will approximate the spectrum by a subset of $T_\tau$.

For every $\lambda \in T_\tau$, we can use the algorithm of Corollary~\ref{approximability} to compute an approximation $\tilde\rho_{H}(z, \tau/4)$ to $\rho_H(z)$.
We define our approximation $\tilde\Gamma(H,\tau)$ of the spectrum as
\begin{align}
\label{eqDefinitionTildeGamma}
  \tilde\Gamma(H,\tau) := \Big\{\,\lambda \in T_\tau \;\Big|\; \tilde\rho_{H}(\lambda, \tau/4) < \tau/2\;\Big\}\,.
\end{align}
We will now show that $d_H(\tilde\Gamma(H,\tau), \Spec(H)) \leq \tau$ in Hausdorff distance.

First, let $\lambda \in \Spec(H)$. Then $\rho_H(\lambda)=0$. Because we have chosen $T_\tau$ to have covering radius $\tau/4$ inside $B_M(0)$, there exists a point $\lambda' \in T_\tau$ with $|\lambda - \lambda'| < \tau/4$. Because $\rho_H$ is Lipschitz continuous with Lipschitz constant~$1$ (see, for example, \cite{noteOnHausdorff}, Lemma 2.1), it follows that $|\rho_H(\lambda)-\rho_H(\lambda')| < \tau / 4$. By Corollary~\ref{approximability}, we have  $|\rho_H(\lambda')- \tilde \rho_H(\lambda', \tau/4)| < \tau/4$. The triangle inequality thus implies
\[
\tilde \rho_{H}(\lambda',\tau/4) \leq \big|\tilde \rho_{H}(\lambda',\tau/4) - \rho_H(\lambda')\big| + \big|\rho_H(\lambda')- \rho_H(\lambda)\big| + \rho_H(\lambda) < \tau/2\,,
\]
where we have used $\rho_H(\lambda)=0$. Hence,  $\lambda' \in  \Gamma_k(H)$. Because $\lambda$ was arbitrary, we have $d(\lambda,  \tilde\Gamma(H,\tau)) \leq \tau/2$ for any $\lambda \in \Spec(H)$.

On the other hand, for any $\lambda \in  \tilde\Gamma(H, \tau)$, the condition $\tilde \rho_{H}(\lambda,\tau/4) < \tau/2$ in~\eqref{eqDefinitionTildeGamma}
implies that $\rho_H(\lambda) < \tau$ by Corollary~\ref{approximability}. But for normal operators, $\rho_H(\lambda) = d(\lambda, \Spec(H))$, so we have $d(\lambda, \Spec(H)) \leq \tau$.

Summing up, for any $\lambda \in \Spec(H)$ we have $d(\lambda, \tilde\Gamma(H,\tau)) \leq \tau/4$ and for any $\lambda' \in \tilde\Gamma(H,\tau)$ we have $d(\lambda', \Spec(H)) \leq \tau$. Thus, we have shown that $d_\mathrm{H}(  \tilde\Gamma(H,\tau), \Spec(H)) \leq \tau$. The proof of Theorem~\ref{theoremComputabilitySpectrum} then follows simply by choosing $\Gamma_k(H) := \tilde\Gamma(H, 2^{-k})$.
\end{proof}

\section{Computability of the pseudospectrum}
\label{sec-computability-pseudospectrum}

In this section, we will prove Theorem~\ref{theoremComputabilityPseudospectrum}, the computability of the $\varepsilon$-pseudospectrum for potentially non-normal operators and for $\varepsilon > 0$.
A related result about pseudospectra is \cite{noteOnHausdorff}, in which it is shown that the pointwise convergence of the lower norm and the lower norm of the adjoint implies the convergence of the $\varepsilon$-pseudospectrum in Hausdorff distance. In the following, we show an analogous result about computability, namely that the $\varepsilon$-pseudospectrum can be computed in Hausdorff distance if arbitrarily extact estimates of the lower norm are possible. We  do not require computation of the lower norm of the adjoint. In another related result \cite{discreteGroups}, the constancy of the spectrum in minimal dynamical systems has also been extended to pseudospectra, which might provide an alternative avenue for computing pseudospectra rigorously.

As in the proof of Theorem~\ref{theoremComputabilitySpectrum}, we will choose a sufficiently fine grid $T_\tau$ on which to compute the approximations $\tilde \rho_{H}(\lambda, \tau)$. However, in the non-normal case, the scale $\tau$ cannot be determined \textit{a priori}.
Instead, we will simultaneously decrease the maximum error $\tau$ of the approximation $\tilde \rho_H(\lambda, \tau)$ and the spacing of the grid $T_\tau$ on which it is computed until a certain condition is met.

As a first step, Proposition~\ref{propSRU} in Subsection~\ref{subsec-condition} will show that a certain easily computable criterion is sufficient to guarantee that the spectrum is approximated well enough for a given~$\tau$.

\subsection{Approximating  the pseudospectrum}
\label{subsec-condition}

To compute the $\varepsilon$-pseudospectrum, we will evaluate $\tilde\rho_H(\lambda, \tau)$ for all $\lambda \in T_\tau$, where $T_\tau$ is the following lattice:
\begin{align}
\label{eqDefinitionTtauPseudo}
T_\tau &:= \tau\sqrt{2}\mathbb{Z}^2 \,\cap\, B_{M + \varepsilon}(0) \subseteq \mathbb{C} \, .
\end{align}
This definition is very close to \eqref{eqDefinitionTtau}, except that the factor $\frac14$ is not necessary in the present case. We then decompose $T_\tau$ into the following three disjoint sets:
\begin{align}
\label{eqDefSTau}
S_\tau &:= \{\, \lambda \in T_\tau \;|\; \tilde\rho_{H}(\lambda,\tau) < \varepsilon - \tau\,\}\,,\\
\label{eqDefRTau}
R_\tau &:= \{\, \lambda \in T_\tau \;|\; \tilde\rho_{H}(\lambda, \tau) > \varepsilon + 2\tau \,\}\,,\\
\label{eqDefUTau}
U_\tau &:= T_\tau \;\backslash\; (S_\tau \cup R_\tau) \, .
\end{align}
It is easy to see, using Corollary~\ref{approximability}, that all points $S_\tau$ are certainly in $\Spec_\varepsilon(H)$, while all points in $R_\tau$ are certainly not in $\Spec_\varepsilon(H)$. The following proposition shows a simple condition on $S_\tau$, $R_\tau$ and $U_\tau$ that, if fulfilled, guarantees that $S_\tau$ will be a good approximation of the spectrum.

\begin{proposition}
\label{propSRU}
Let $H$ be an \emph{flc} operator and $\delta,\tau > 0$. Let $S_\tau$ and $U_\tau$ be the sets defined in \eqref{eqDefSTau} and \eqref{eqDefUTau}. Then, if
\begin{align}
\label{conditionSRU}
d(\lambda, S_\tau) < \delta - \tau \qquad\text{for all }\lambda \in U_\tau\,,
\end{align}
we have
\begin{align*}
 d_H(S_\tau, \Spec_\varepsilon(H)) < \delta\,.
\end{align*}
\end{proposition}
\begin{proof}
By definition of $S_\tau$, we have $S_\tau \subseteq \Spec_\varepsilon(H)$. Thus, the inequality for the Hausdorff distance reduces to showing that $d(\lambda, S_\tau) \leq \delta$ for all $\lambda \in \Spec_\varepsilon(H)$.

Let $\lambda \in \Spec_\varepsilon(H)$ be given. It is easy to show that $\lambda \in B_{M + \varepsilon}(0)$. The lattice $T_\tau$ was chosen to have covering radius $\tau$ in $B_{M + \varepsilon}(0)$, so we can find a $\lambda' \in T_\tau$ with $d(\lambda, \lambda') \leq \tau$. Consider now the additional set $R_\tau$ defined in \eqref{eqDefRTau}. It is clear from the definitions of $S_\tau$, $R_\tau$ and $U_\tau$, \eqref{eqDefSTau}, \eqref{eqDefRTau} and \eqref{eqDefUTau} respectively, that they are a disjoint decomposition of $T_\tau$. Thus, $\lambda'$ must be in exactly one of these sets.

Of the three options, $\lambda' \in R_\tau$ is actually impossible. That is because we have $\lambda \in \Spec_\varepsilon(H)$, so $\rho_H(\lambda) \leq \varepsilon$. The Lipschitz continuity of $\rho_H$ implies that $\rho_H(\lambda') \leq \varepsilon + \tau$. From Corollary~\ref{approximability}, we then get $\tilde\rho_H(\lambda', \tau) \leq \varepsilon + 2 \tau$. This means $\lambda' \notin R_\tau'$.

If we have $\lambda' \in S_\tau$, then $d(\lambda, S_\tau) \leq \tau$. But the algorithm always chooses $\tau \leq \delta$, so we have $d(\lambda, S_\tau) \leq \delta$ in this case.

Finally, suppose that $\lambda' \in U_\tau$. By the hypothesis \eqref{conditionSRU}, we have $d(\lambda', S_\tau) < \delta - \tau$. Moreover $d(\lambda, \lambda') \leq \tau$, thus we have $d(\lambda, S_\tau) \leq \delta$ also in this case. Therefore $d(\lambda, S_\tau) \leq \delta$ is fulfilled for every $\lambda \in \Spec_\varepsilon(H)$.
\end{proof}

\subsection{The computational algorithm}

Because the sets $S_\tau$ and $U_\tau$ are always finite, condition~\eqref{conditionSRU} from Proposition~\ref{propSRU} can be checked in finite time, given the sets. Because this condition is sufficient to show that $S_\tau$ is within distance~$\delta$ of~$\Spec_\varepsilon(H)$, we can compute the $\varepsilon$-pseudospectrum by simply checking smaller and smaller $\tau > 0$, as in the following algorithm.

\begin{alg}[Approximate the $\varepsilon$-pseudospectrum of $H$ with maximum error $\delta$]\label{algPseudospectrum}~\\\vspace*{-0.3cm}
\begin{algorithmic}
  \State Compute an upper bound $M$ on the norm of $H$
  \State Compute the lattice $T_\tau := \tau\sqrt{2}\mathbb{Z}^2 \,\cap\, B_{M + \varepsilon}(0) $
  \For{$j \gets 1, 2, 3, \dots$}
  \State Set $\tau = 2^{-j}\cdot \delta$
  \State compute $S_\tau$ and $U_\tau$ according to Equations~\eqref{eqDefSTau} and~\eqref{eqDefUTau}
  \If{condition~\eqref{conditionSRU} holds}
  \State \Return $S_\tau$
  \EndIf
  \EndFor
\end{algorithmic}
\end{alg}
Proposition~\ref{propSRU} shows that if  Algorithm~\ref{algPseudospectrum} terminates, the result  fulfills $d_H(S_\tau, \Spec_\varepsilon(H)) < \delta$. To complete the proof of Theorem~\ref{theoremComputabilityPseudospectrum} it remains to show that Algorithm~\ref{algPseudospectrum} always terminates. In other words, we need to show that for any \emph{flc} operator~$H$, as well as $\varepsilon$ and $\delta$, condition~\eqref{conditionSRU} is fulfilled for $\tau$ small enough. This will be proven in the following subsection.

\subsection{Proof that the algorithm terminates}

\begin{lemma}
\label{lemmaTermination}
Let an \emph{flc} operator $H$ and $\varepsilon,$ $\delta > 0$ be given. Then there exists a $\tau' > 0$ such that for all $0 < \tau < \tau'$, the sets $S_\tau$ and $U_\tau$ defined in equations~\eqref{eqDefSTau} and~\eqref{eqDefUTau} fulfill the condition~\eqref{conditionSRU} in Proposition~\ref{propSRU}.
\end{lemma}
\begin{proof}
By Lemma~\ref{lemmaPreTau}, there exists a $\tau' > 0$ such that for all $0 < \tau < \tau'$,
\begin{align}
 d_H(\Spec_{\varepsilon-3\tau'}(H), \Spec_{\varepsilon+3\tau'}(H)) < \delta\,.
 \label{eqThreeTau}
\end{align}
(We can choose a smaller $\delta$ in Lemma~\ref{lemmaPreTau} to obtain the strict inequality.)

Now let $\lambda \in U_\tau$ be arbitrary. By definition of $U_\tau$, we have
\begin{align}
 \tilde\rho_H(\lambda, \tau) \leq \varepsilon + 2 \tau\,,
\end{align}
because otherwise we would have $\lambda \in R_\tau$. From Corollary \ref{approximability}, we can deduce
\begin{align}
 \rho_H(\lambda) < \varepsilon + 3 \tau\,.
\end{align}
Now by \eqref{eqThreeTau}, we know that there exists a $\lambda'' \in \Spec_{\varepsilon-3\tau}(H)$ with $d(\lambda, \lambda'') < \delta$. This $\lambda''$ hence fulfills
\begin{align}
\label{whatLambdaPPfulfills}
 \rho_H(\lambda'') \leq \varepsilon - 3 \tau\,.
\end{align}
Now we will set
\begin{align}
 \lambda' = \lambda'' + \tau \frac{\lambda - \lambda''}{|\lambda - \lambda''|}\,.
\end{align}
The point $\lambda'$ lies on the line connecting $\lambda$ and $\lambda''$, at a distance $\tau$ from $\lambda''$ in the direction of $\lambda$. Clearly, we have
\begin{align}
 d(\lambda', \lambda) < \delta - \tau\,,
 \label{eqLambdaDist}
\end{align}
because $d(\lambda, \lambda'') < \delta$. On the other hand, because $\rho_H$ is Lipschitz continuous with constant $\leq 1$ and $d(\lambda', \lambda'') = \tau$, we have
\begin{align}
 \rho_H(\lambda') \leq \rho_H(\lambda'') + \tau \leq \varepsilon - 2 \tau\,.
\end{align}
where we used \eqref{whatLambdaPPfulfills} in the second inequality. Corollary~\ref{approximability} again gives
\begin{align}
 \tilde\rho_H(\lambda', \tau) \leq \varepsilon - \tau\,.
\end{align}
It follows that $\lambda' \in S_\tau$. We have shown $d(\lambda', \lambda) < \delta - \tau$ in \eqref{eqLambdaDist}, so $d(\lambda, S_\tau) < \delta - \tau$. But $\lambda \in U_\tau$ was arbitrary, proving the Lemma.
\end{proof}

\begin{proof}[Proof of
Theorem~\ref{theoremComputabilityPseudospectrum}]
The solution to the computational problem $(\Omega_{\rm ps},$ $\Lambda_{\rm ps},$ $(\mathcal M, d),$ $\Xi_{\rm ps})$ is given by the BSS algorithm $\Gamma_k(H, \varepsilon)$ that is just the application of Algorithm~\ref{algPseudospectrum} to the given $H$, $\varepsilon$ and $\delta = 2^{-k}$.

By Lemma~\ref{lemmaTermination}, there exists a $\tau' > 0$ such that condition~\eqref{conditionSRU} is fulfilled for all $0 < \tau < \tau'$. Thus, Algorithm~\ref{algPseudospectrum} will terminate after at most $\log_2(\max(1, \delta / \tau'))$ iterations. After the algorithm has terminated, the sets $S_\tau$ and $U_\tau$ fulfill condition \eqref{conditionSRU}, and therefore, by Proposition~\ref{propSRU}, we have
\begin{align*}
d_H(\Gamma_k(H, \varepsilon), \Xi_{\rm ps}(H, \varepsilon)) = d_H(S_\tau, \Spec_\varepsilon(H)) < \delta = 2^{-k}\,.
\end{align*}
This completes the proof of Theorem~\ref{theoremComputabilityPseudospectrum}.
\end{proof}

\section*{Acknowledgements}

We would like to thank Marko Lindner for sharing with us the unpublished manuscript \cite{lindner_unpublished} and for valuable suggestions. We are grateful to Siegfried Beckus,  Anuradha Jagannathan,  Johannes Kellendonk, and Lior Tenenbaum for helpful comments. The work of M.M.\ has been supported by a fellowship of the Alexander von Humboldt Foundation during his stay at the University of T\"ubingen, where this work initiated.
M.M. gratefully acknowledges the support of PNRR Italia Domani and Next Generation EU through the ICSC National Research Centre for High Performance Computing, Big Data and Quantum Computing and the support of the MUR grant Dipartimento di Eccellenza 2023–2027. S.T.\ acknowledges financial support by the Deutsche Forschungsgemeinschaft (DFG, German Research Foundation) – TRR 352 – Project-ID 470903074.

\appendix
\section{Elementary lemmas}
\label{sec-elementary}

This Appendix collects several Lemmas with simple proofs that have been moved here to shorten the main text.

In proving the reduction to finite range in Section~\ref{sec-finite-range}, we used the following lemma, which provides a bound on the sum of polynomially decaying hoppings starting at a given distance.

\decayLemma*

\begin{proof}
Without loss of generality, we can assume that $m>l$. To compute the sum in \eqref{eqDistpowers}, we use the Cauchy-Maclaurin criterion adapted to our setting in same spirit as in \cite{marcelli_2023}. The main idea is to estimate the sum from above by using an integral. First, notice that $d(x,y)\leq \|x-y\|_2 \leq  d(x,y) \sqrt{n}$ where $\|\cdot\|_2$ denotes the usual Euclidean norm in $\mathbb{R}^n$. Then, the function $\mathbb{R}^n \ni y \mapsto d(x,y)^{-(n+\varepsilon)}$ is such that $d(x,y)^{-(n+\varepsilon)}\leq n^{(n+\varepsilon)/2} \|x-y\|_2^{-(n+\varepsilon)}$. For every $y \in \Gamma$ such that $d(x,y) > m$, let $\tilde{B}^{(x)}_{l/4}(y)$ be a closed (possibly rotated) hypercube of side length $l/2$ such that $y$ sits on one of its corner and for every $z \in \tilde{B}^{(x)}_{l/4}(y)$ it holds that $\|x-z\|_2\leq \|x-y\|_2$. Furthermore, by hypothesis on the uniform discreteness of $\Gamma$ we have that $\tilde{B}^{(x)}_{l/4}(y) \cap \Gamma =\{y\}$ for every $y \in \Gamma$ and $\tilde{B}^{(x)}_{l/4}(y) \cap \tilde{B}^{(x)}_{l/4}(z) =\emptyset$ if $y\neq z$. Thus we have
\begin{equation*}
\begin{aligned}
\sum_{\substack{y \in \Gamma \\ d(x,y) > m}} d(x,y)^{-(n+\varepsilon)}  &\leq n^{(n+\varepsilon)/2}  (l/2)^{-n}\sum_{\substack{y \in \Gamma \\ d(x,y) > m}} (l/2)^n\|x-y\|_2^{-(n+\varepsilon)} \\
&\leq n^{(n+\varepsilon)/2}  (l/2)^{-n} \sum_{\substack{y \in \Gamma \\ d(x,y) > m}} \int_{\tilde{B}^{(x)}_{l/4}(y)} \mathrm{d}^n z \; \; \|x-z\|_2^{-(n+\varepsilon)}  \\
&\leq n^{(n+\varepsilon)/2}  (l/2)^{-n} \int_{\|y-x\|_2>\frac{m}{4}} \mathrm{d}^n z \; \; \|x-z\|_2^{-(n+\varepsilon)} \\
& \leq n^{(n+\varepsilon)/2}  (l/2)^{-n} (2 \pi)^{n-1}\int_{\frac{m}{4}}^{+\infty} \mathrm{d}r \; \; r^{n-1} r^{-(n+\varepsilon)}
\, .
\end{aligned}
\end{equation*}
where in the last integral we used the $n$-spherical coordinates. By explicit integration, the proofs is concluded.
\end{proof}

The following is a sort of ``Sandwich Lemma'' for the Hausdorff distance. It proves that if a set $B$ is sandwiched between two sets $A$ and $C$, then from the perspective of any reference set $X$, the intermediate set $B$ cannot be farther away than both of the extremal sets $A$ and $C$.

\begin{lemma}
\label{lemmaABCX}
Let $A, B, C, X$ be four subsets of a metric space $(\mathcal M, d)$ such that $A \subseteq B \subseteq C$. Then the Hausdorff distance $d_\mathrm{H}$ fulfills
\begin{align*}
d_\mathrm{H}(X, B) &\leq \max\left( d_\mathrm{H}(X, A), d_\mathrm{H}(X, C) \right)\,.
\end{align*}
\end{lemma}
\begin{proof}
By definition, we have
\begin{align*}
d_\mathrm{H}(X, B) = \max\left( \sup_{x \in X} d(x, B), \sup_{b \in B} d(b, X)\right)\,.
\end{align*}
Now because $A \subseteq B$, we have
\begin{align*}
d(x, B) \,\leq\, d(x, A)\,\leq\, d_\mathrm{H}(X, A),
\end{align*}
and from $B \subseteq C$, we get
\begin{align*}
\sup_{b \in B} d(b, X) \,\leq\, \sup_{c \in C} d(c, X)\,\leq\, d_\mathrm{H}(C, X).
\end{align*}
Thus we obtain $d_\mathrm{H}(X, B) \leq \max(d_\mathrm{H}(X, A), d_\mathrm{H}(X, C)) = M$.
\end{proof}

\section{Definition of BSS algorithms}
\label{appendixBSS}

Computational problems can be solved by algorithms. One can define classes of algorithms corresponding to different models of computation \cite{benartzi_2020}. In the following, we will use Blum-Shub-Smale (BSS) machines to define the computational steps of our algorithms \cite{blum_1989, blum_1998, turingMeetsNewton,topologicalViewComputationModels}. BSS machines are similar to Turing machines \cite{turing_1937} but can store arbitrary real numbers (or elements of any other ring) on their tape. A BSS~machine is characterized by a certain set of states and rules. Any such machine defines a mapping
\begin{gather*}
 f_{\text{BSS}} : \mathbb{R}^\infty \to \{\fbox{\text{no-halt}}\} \cup \mathbb{R}^\infty\,,
\end{gather*}
where $\mathbb{R}^\infty = \mathbb{R}^0 \,\dot\cup\, \mathbb{R}^1 \,\dot\cup\, \mathbb{R}^2 \,\dot\cup\, \dots$ is the set of finite real sequences and \fbox{no-halt} is a special symbol. To determine the value of the function $f_{\text{BSS}}$ at a specific $x \in \mathbb{R}^\infty$, the elements of $x$ are written on the tape, and the value of the function $f_{\text{BSS}}$ will be the numbers on the tape after the machine halts, or \fbox{no-halt} if the machine does not halt on the given inputs. We say that a BSS machine \textit{always halts} if $f(A) \neq \fbox{\text{no-halt}}$ for all $A \in \mathbb{R}^\infty$.

\begin{definition}
Let $(\Omega, \Lambda, (\mathcal M, d), \Xi)$ be a computational problem and let $F_{\text{embed}} : \mathbb{R}^\infty \to \mathcal M$ be a function with dense image. (The function $F_{\text{embed}}$ is used to parameterize the solution space.) Furthermore, suppose that $\Lambda$ is countable and $(f_i)_{i \in \mathbb N}$ is an enumeration of $\Lambda$.

A \textit{BSS algorithm} $\Gamma$ for the problem $(\Omega, \Lambda, (\mathcal M, d), \Xi)$ can be described by a sequence of BSS functions $(c_k)_{k \in \mathbb{N}_0} : \mathbb{R}^k \to \mathbb{R}^\infty$, all of which halt on any input. Furthermore, we require that for any $x \in \mathbb{R}^k$, the vector $a := c_k(x)$ has at least one element, and furthermore
\begin{itemize}
 \item If $a_0 \neq 0$, then $a \in \mathbb{R}^2$, and $a_1$ is an integer.
 \item If $a_0 = 0$, then $a$ has an odd number of elements.
\end{itemize}
\end{definition}

The purpose of these conditions on $a$ is that we wish to reserve the first number of the result $a_0$ to indicate whether the algorithm has terminated in step $k$ or whether a further evaluation is requested. If $a_0 \neq 0$, then $a_1 \in \mathbb{Z}$ shall determine the evaluation function that is requested. If $a_0 = 0$, then the execution terminates, and the following pairs of numbers shall determine the endpoints of the intervals whose union is the result of the computation, parameterized via $F_{\text{embed}}$.

To evaluate a BSS algorithm on a value $x \in \Omega$, we define a series of functions
\begin{align*}
d_k : \Omega \to \mathcal M \,\dot\cup\, \mathbb{R}^{k+1},\quad k \in \mathbb N\,.
\end{align*}
which follow the instructions given by $c_k(x)$. As a first step, we let
\begin{align*}
 d_0(A) = c_0\,,
\end{align*}
where $c_0 \in \mathbb{R}^\infty$ since $c_0$ has zero arguments.

Now for all $k > 0$, the value of $d_k(A)$ depends on the value of $d_{k-1}(A)$. If $d_{k-1}(A) \in \mathbb{R}^{k}$, then the algorithm has not terminated yet, but has evaluated $(k+1)$ evaluation functions. To decide how to proceed in the next step, we then use the function $c_k$. Let
\begin{align*}
 a = (a_0, \dots, a_m) = c_k(d_{k-1}(A))\,.
\end{align*}
If $a_0 \neq 0$, then $c_k$ has determined that another evaluation is necessary. In that case, $m = 1$, and $a_1$ contains the index of the function to be evaluated. Therefore, in this case, we let
\begin{align*}
 d_k(A) = (b_1, \dots, b_k, f_{a_1}(A))\,.
\end{align*}
where $(b_0, \dots, b_{k-1}) = d_{k-1}(A)$, thus appending the result of the requested evaluation to the previous ones.

If on the other hand $a_0 \neq 0$, then $c_k$ has determined that the previous evaluations were sufficient to estimate the result, and the values $a_1, \dots, a_{N}$ contain a representation of the result. In this case, therefore, we let
\begin{align*}
 d_k(A) = F_{\text{embed}}(a_1, \dots, a_N)\,.
\end{align*}

Finally, if $d_{k-1}(A) \in \mathcal M$, then a result has already been found in a previous step, so that we have to perform no further computation and can simply define
\begin{align*}
 d_k(A) = d_{k-1}(A)\,.
\end{align*}

Again, we say that the algorithm \textit{terminates} for an argument $A \in \Omega$ if $d_k(A) \in \mathcal M$ for some $k \in \mathbb N$. If an algorithm terminates for all $A \in \Omega$, we can define a map $\Gamma : \Omega \to \mathcal M$ by mapping $A$ to $d_k(A)$ for the first $k \in \mathbb N$ for which $d_k(A) \in \mathcal M$. If we speak of an \textit{algorithm} $\Omega \to \mathcal M$ in the following, we always mean an algorithm that terminates for all $A \in \Omega$, and we will usually identify the algorithm with its associated map $\Gamma:\Omega\to\mathcal M$.

\section{Evaluation functions for \textit{flc} operators}
\label{appendixEvaluationFunctions}

In this section, we will describe how to augment the set of evaluation functions $\Lambda$ in a way that allows algorithms to make use of the \textit{flc} structure for the \textit{flc} spectral and pseudospectral problems defined in Definitions~\ref{definitionSpectralProblem} and \ref{definitionPseudospectralProblem}.

The evaluation functions have to represent a given \textit{flc} operator by a set of real functions. Let $H$ be an \textit{flc} discrete operator. As a first step, for every $L \in \mathbb{N}$, let $(x_{L,i})_{i = 1, \dots, n_{\textrm{patch}}(L)}$ be an enumeration of the local patches at scale $L$, so that for every $y \in \Rn$, there is exactly one $n \in \{1, \dots, n_{\textrm{patch}}(L)\}$ such that $H$ has equivalent action on $B_L(y)$ and $B_L(x_{L,i})$.

Now for every local patch $B_L(x_m)$, since $\Gamma$ is uniformly discrete, the set $B_L(x_m) \cap \Gamma$ is finite. For any $L > 0$, $m \in \{1, \dots, n_{\rm patch}(L)\}$, let $(p_{L,m,k})$ be the $k$-th point in $B_L(x_m)$, according to some enumeration of the points. Using these definitions, we can now define our conditions on a set of evaluation functions that captures the finite local complexity structure.

\begin{condition}
\label{definitionEvaluationFunctions}
Let
\begin{align*}
 \mathcal{I} := &\;(\{1\} \times \mathbb{N}) \;\dot\cup\; (\{2\} \times \mathbb{N} \times \mathbb{N}) \;\dot\cup\; (\{3\} \times \mathbb{N}\times\mathbb{N}\times\mathbb{N}\times\mathbb{N}) \\
 &\;\dot\cup\; (\{4\} \times \mathbb{N}\times\mathbb{N}\times\mathbb{N}\times\mathbb{N})\,.
\end{align*}
Then a family of functions $(f_i)_{i \in \mathcal I}$ fulfills our conditions for a set of evaluation functions if:
\begin{itemize}
 \item $f_{1,L}(H)$ is the number of local patches, up to equivalent action, of $H$ at scale $L$, defined above as $n_{\textrm{patch}}(L)$.
 \item $f_{2,L,m}(H)$ is the number of points in $\Gamma \cap B_L(x_m)$, defined above as $n_{\rm point}(L, m)$.
 \item $f_{3,L,m,k,l}(H)$ is the matrix element $\langle p_k , H_{B_L(x_m)}  p_l \rangle$, where $p_k$ and $p_l$ are the points in $\Gamma \cap B_L(x_m)$ according to the above enumeration.
 \item $f_{4,L,n,k,i}(H)$ is $i$-th coordinate of $p_{L,n,k}$, defined above as the $k$-th point in the $m$-th patch at scale $L$, or zero if $n > n_{\rm patch}(L)$ or $i > n$ or $k > n_{\rm patch}(L)$.
 \item $f_{5}(H)$ is a bound on the norm of $H$.
 \item $f_{6}(H) =: M$ and $f_{7}(H) =: \varepsilon$ are real numbers such that
 \begin{align*}
|H_{xy}| \leq M \, d(x,y)^{-(n+\varepsilon)}\,,
 \end{align*}
 for all $x, y \in \Gamma$, as in Definition~\ref{definition-short-range}.
\end{itemize}
The evaluation functions $f_{2, L, m}(H)$, $f_{3, L, m, k, l}(H)$, and $f_{4, L, m, k, i}(H)$ are defined to be zero if any index is out of bounds, that is if $m > n_{\rm patch}(L)$ or $k > n_{\rm point}(L, m)$ or $l > n_{\rm point}(L, m)$ or $i > n$.
\end{condition}

\bibliographystyle{amsplain}
\bibliography{main}

\end{document}